\documentclass[11pt]{amsart}

\usepackage[margin=1.15in]{geometry}
\usepackage{amsmath,amssymb,amsthm,mathtools,bm,mathrsfs}
\usepackage{booktabs,array,longtable}
\usepackage{enumitem}
\usepackage{microtype}
\usepackage[hidelinks]{hyperref}
\usepackage{xurl}
\usepackage{orcidlink}

\allowdisplaybreaks[2]
\numberwithin{equation}{section}

\newtheorem{theorem}{Theorem}[section]
\newtheorem{lemma}[theorem]{Lemma}
\newtheorem{proposition}[theorem]{Proposition}
\newtheorem{corollary}[theorem]{Corollary}
\theoremstyle{definition}
\newtheorem{definition}[theorem]{Definition}
\theoremstyle{remark}
\newtheorem{remark}[theorem]{Remark}

\newcommand{\Ric}{\operatorname{Ric}}
\newcommand{\Kill}{\mathcal B}
\newcommand{\Bmix}{\mathsf{B}}
\newcommand{\Isom}{\operatorname{Isom}}
\newcommand{\Inn}{\operatorname{Inn}}
\newcommand{\Ad}{\operatorname{Ad}}
\newcommand{\tr}{\operatorname{tr}}
\newcommand{\cof}{\operatorname{cof}}
\newcommand{\Sym}{\operatorname{Sym}}
\newcommand{\rank}{\operatorname{rank}}
\newcommand{\diag}{\operatorname{diag}}

\newcommand{\Frob}{\mathrm F}
\newcommand{\Kop}{\mathscr K}
\newcommand{\Lop}{\mathscr L}
\newcommand{\Ccal}{\mathcal C}
\newcommand{\Hc}{\mathscr H}
\newcommand{\Th}{\vartheta}
\newcommand{\ip}[2]{\left\langle #1,#2\right\rangle}
\newcommand{\norm}[1]{\left\lVert #1\right\rVert}
\newcommand{\sym}{\operatorname{sym}}

\title[Left-invariant Einstein metrics on $S^3\times S^3$]
{A complete classification of left-invariant Einstein metrics on $S^3\times S^3$}

\author{Sixuan Gu}
\address{Institute of Mathematical Sciences, The Chinese University of Hong Kong, Hong Kong, China}
\email{sixuangu@link.cuhk.edu.cn}

\author[Wei Qi]{Wei Qi\,\orcidlink{0009-0004-5794-4907}}
\address{Department of Physics, The Ohio State University, Columbus, OH 43210, USA}
\email{qi.673@osu.edu}

\date{}

\begin{document}

\begin{abstract}
We complete the classification of compact simply connected
homogeneous Einstein manifolds in dimension six by resolving
the remaining cases for left-invariant Einstein metrics on
$G=\mathrm{SU}(2)\times\mathrm{SU}(2)\cong S^3\times S^3$.
Previous work leaves two cases for the isotropy group $K$,
namely $K=\{e\}$ and $K\cong\mathbb Z_2$.
We first rule out $K=\{e\}$. Then we show that any $\mathbb Z_2$ symmetry generated by an
inner involution $\sigma$ with $\tr\sigma=-2$ necessarily
extends to a $\mathbb Z_2\times\mathbb Z_2$ symmetry.
Together with the previously known classification results,
these theorems imply that every left-invariant Einstein metric
on $G$ is, up to homothety and isometry, either the standard
product metric $g_{\rm can}$ or the Jensen nearly K\"ahler
metric $g_{\rm NK}$.
\end{abstract}

\maketitle
\setcounter{tocdepth}{1}
{\small\tableofcontents}

\section{Introduction and main results}

The classification of compact simply connected homogeneous
Einstein manifolds in dimension six has been reduced to the
classification of left-invariant Einstein metrics on
$\mathrm{SU}(2)\times\mathrm{SU}(2)\cong S^3\times S^3$;
see \cite{NR99,NR,BCHL}.
In this paper we resolve the remaining cases,
thereby completing the classification.

\subsection{Geometric conventions}

Let $G=\mathrm{SU}(2)\times\mathrm{SU}(2)$ and
$\mathfrak g=\mathfrak{su}(2)\oplus\mathfrak{su}(2)$.
A left-invariant metric is determined by an inner product on
$\mathfrak g$, and it is Einstein when $\Ric_g=\lambda g$.
We classify metrics up to homothety and isometry: $g_1$ and $g_2$
are equivalent when $g_2=c\varphi^*g_1$ for some $c>0$ and
some diffeomorphism $\varphi:G\to G$.

Write $\Kill$ for the Killing form and equip each factor with
$\langle\cdot,\cdot\rangle_0=-\tfrac12\Kill$.
Fix a Lie-algebra isometry $\iota$ from the first factor to the second.
The standard product metric $g_{\rm can}$ and a representative of the
Jensen metric $g_{\rm NK}$ are specified at the identity by
\begin{align*}
(g_{\rm can})_e\bigl((X,Y),(X',Y')\bigr)
&=\langle X,X'\rangle_0+\langle Y,Y'\rangle_0,\\
(g_{\rm NK})_e\bigl((X,Y),(X',Y')\bigr)
&=2\langle X,X'\rangle_0+2\langle Y,Y'\rangle_0
-\langle\iota X,Y'\rangle_0-\langle Y,\iota X'\rangle_0.
\end{align*}
The second metric is the homogeneous nearly K\"ahler metric
\cite{BCHL, Jensen}. Its isometry class does not depend on $\iota$,
since any two choices differ by an inner automorphism.

The symmetry group used throughout the proof is the \emph{inner isotropy}
\begin{equation}\label{eq:inner-isotropy-definition}
K(g):=\Isom(G,g)\cap\Inn(G),
\end{equation}
where $\Inn(G)$ consists of the conjugations $x\mapsto axa^{-1}$.
Every such map fixes the identity. Its differential identifies
$\Inn(G)$ with $\Ad(G)\cong SO(3)\times SO(3)$, and $K(g)$
with the stabilizer of $g_e$ in this compact group. In particular,
$K(g)$ is a compact Lie group. It is the inner-automorphism part of
$\Isom(G,g)_e$, not necessarily the full isotropy group.

For $\sigma\in\Ad(G)$, $\tr\sigma$ denotes its trace on the
six-dimensional Lie algebra $\mathfrak g$. An \emph{inner involution}
is a nonidentity element satisfying $\sigma^2=I$.
A nonidentity involution in $SO(3)$ is a rotation by $\pi$, of trace
$-1$. Thus the nonidentity involutions in $SO(3)\times SO(3)$ have
trace $2$ or $-2$, according as one or both factors are nonidentity.
In the trace $-2$ case, the $+1$ and $-1$ eigenspaces have dimensions
two and four, respectively. Conjugating by $\Ad(G)$ corresponds to
an inner change of the oriented bases in the two factors.

\subsection{Previous results and the two remaining cases}

Previous work reduces the classification of left-invariant Einstein metrics on
\(\mathrm{SU}(2)\times\mathrm{SU}(2)\) to two residual cases. If the inner-isometry isotropy $K(g)$
has positive dimension, then \(K(g)\) contains a circle subgroup, and \cite[Theorem~1]{BCHL}, summarizing the result of \cite{NR}, implies that \(g\) is, up to homothety and isometry, either the standard product metric \(g_{\rm can}\) or the Jensen nearly K\"ahler metric \(g_{\rm NK}\). Belgun--Cort\'es--Haupt--Lindemann \cite{BCHL} proved the same conclusion when \(K(g)\) contains a nontrivial finite subgroup of \(\Ad(G)\) that is not isomorphic to \(\mathbb Z_2\); see \cite[Theorem~2]{BCHL}.

It therefore remains to understand the cases \(K(g)\cong\mathbb Z_2\) and \(K(g)=\{e\}\). For \(K(g)\cong\mathbb Z_2\), BCHL distinguish the two conjugacy classes of inner involutions by their traces. The trace \(2\) case is classified in \cite[Proposition~7]{BCHL}. In the trace \(-2\) case, they reduce the Einstein equation to an explicit polynomial system, but leave the final classification open; see \cite[Section~3.2]{BCHL}.
For trivial inner isotropy,
$
K(g)=\{e\},
$
we prove that this case cannot occur for a left-invariant Einstein metric. Together with the resolution of the trace $-2$ $\mathbb Z_2$ case, this completes the classification.

\subsection{Main theorems}

The first result supplies the missing symmetry automatically.
\begin{theorem}[Automatic inner symmetry]\label{thm:auto}
Every left-invariant Einstein metric on
$G=\mathrm{SU}(2)\times\mathrm{SU}(2)$ is preserved by a nonidentity inner automorphism.  Equivalently,
\[
\Ric_g=\lambda g
\qquad\Longrightarrow\qquad
K(g)\neq\{e\}.
\]
\end{theorem}

The second result is a symmetry enhancement: an assumed trace $-2$ involution is forced to lie in a Klein four subgroup.

\begin{theorem}[Symmetry enhancement for a trace $-2$ involution]\label{thm:z2upgrade}
Let $g$ be a left-invariant Einstein metric on $G$.  Suppose that $g$ is preserved by an inner involution
\[
\sigma\in\Ad(G),\qquad \sigma^2=I,\qquad \sigma\neq I,
\qquad \tr\sigma=-2.
\]
Then the subgroup of $\Ad(G)$ preserving $g$ contains a Klein four subgroup that contains $\sigma$:
\[
\langle\sigma,\tau\rangle\cong\mathbb Z_2\times\mathbb Z_2
\]
for some inner involution $\tau$ commuting with $\sigma$.
\end{theorem}
A \texttt{Lean 4} formalization for Theorem~\ref{thm:auto} and Theorem~\ref{thm:z2upgrade} is available, using \texttt{mathlib} \cite{lean4,mathlib}, at \url{https://github.com/wqirocks/S3xS3-LI-Einstein}. 

Combining these two statements with the cases already classified in the literature gives the final result.

\begin{theorem}[Complete classification]\label{thm:classification}
Every left-invariant Einstein metric on
$S^3\times S^3\cong\mathrm{SU}(2)\times\mathrm{SU}(2)$ is, up to homothety and isometry, either the standard product metric $g_{\rm can}$ or the Jensen metric $g_{\rm NK}$.
\end{theorem}
Together with the classification results of
Nikonorov--Rodionov \cite{NR99,NR},
Theorem~\ref{thm:classification} completes the classification
of compact simply connected homogeneous Einstein manifolds
in dimension six.

\subsection{Outline of the proof}

Sections~\ref{sec:linear}--\ref{sec:einstein-crit} establish global
coordinates and the scalar-curvature equations. The radial cofactor
identity gives the first spectral bound in Section~\ref{sec:first-bound};
a factor interchange gives the second in Section~\ref{sec:swap}.
After the rank reductions, the off-diagonal Einstein equations restrict
the possible supports to four types. Their exclusion proves
Theorem~\ref{thm:auto} in Section~\ref{sec:auto-proof}.
Sections~\ref{sec:z2-setup}--\ref{sec:z2-upgrade} prove
Theorem~\ref{thm:z2upgrade}: the paired equations determine the signs,
a square identity bounds the anisotropy, and a diagonal
Einstein difference excludes the fully mixed anisotropic case.
The remaining parameter cases yield the required Klein four subgroup.
Section~\ref{sec:synthesis} combines these results with the previously
classified cases. The appendices relate both coordinate descriptions
to the polynomial systems in \cite{BCHL}.

\section{Linear-algebraic preliminaries}\label{sec:linear}

Write $M_n(\mathbb R)$ for the real $n\times n$ matrices,
$\Sym_n$ for the symmetric matrices, and $\Sym_n^+$ for their
positive-definite cone. For symmetric $A,B$, write $A\succeq B$
when $A-B$ is positive semidefinite and $A\succ B$ when it is
positive definite; reversed inequalities have the analogous meaning.
In particular, $H\prec4I$ means that all eigenvalues of $H$ are
strictly less than $4$, and $H>0$ means $H\in\Sym_n^+$.

The Frobenius inner product and norm on matrices are
\[
\ip{X}{Y}=\tr(X^TY),\qquad \norm{X}_{\Frob}^2=\ip{X}{X}.
\]
Whenever a gradient of a scalar-valued matrix function is written without further qualification, it is taken with respect to this inner product.  If $F$ is differentiable, $DF_A[H]$ denotes its derivative at $A$ in the direction $H$.  For $H\in\Sym_n^+$, $H^{1/2}$ denotes the unique positive-definite square root of $H$.
Write
\[
\mathfrak g=\mathfrak{su}(2)\oplus\mathfrak{su}(2).
\]
On each factor we use the background inner product
\[
\ip{X}{Y}_0=-\frac12\Kill(X,Y),
\]
where $\Kill$ is the Killing form.  Choose oriented $\ip{\cdot}{\cdot}_0$-orthonormal bases
\[
(E_1,E_2,E_3),\qquad(F_1,F_2,F_3)
\]
satisfying
\[
[E_i,E_j]=\varepsilon_{ijk}E_k,
\qquad [F_i,F_j]=\varepsilon_{ijk}F_k,
\qquad [E_i,F_j]=0.
\]
Here $\varepsilon_{ijk}$ is the alternating symbol with $\varepsilon_{123}=1$.  The same orientation determines the usual cross product on each copy of $\mathbb R^3$ after identifying $E_i$ or $F_i$ with the standard basis.  Later, ``$(i,j,k)$ cyclic'' always means one of $(1,2,3)$, $(2,3,1)$, or $(3,1,2)$.

For $A\in M_3(\mathbb R)$, define its cofactor matrix by
\[
(\cof A)_{ij}=(-1)^{i+j}\det A[\widehat i\mid\widehat j],
\]
where $A[\widehat i\mid\widehat j]$ is obtained by deleting row $i$
and column $j$. This convention gives
\[
(\cof A)A^T=A^T(\cof A)=(\det A)I.
\]
For invertible $A$, $\cof A=(\det A)A^{-T}$; the signed-minor
definition applies also when $A$ is singular. With respect to the
Frobenius inner product, $\cof A$ is the gradient of the determinant:
\[
D(\det)_A[H]=\ip{\cof A}{H}.
\]

\begin{lemma}[Differential identities for the cofactor]\label{lem:cof-basic}
For all $A,H\in M_3(\mathbb R)$:
\begin{enumerate}[label=\textup{(\roman*)}]
\item $D(\det)_A[H]=\ip{\cof A}{H}$;
\item $\Ccal_A:=D(\cof)_A$ is self-adjoint for the Frobenius inner product;
\item
$
\Ccal_A(H)A^T=\ip{\cof A}{H}I-(\cof A)H^T;
$
\item $\Ccal_A(A)=2\cof A$.
\end{enumerate}
\end{lemma}

\begin{proof}
The first identity is the standard differential formula for the determinant.  Since $\cof A=\nabla(\det)(A)$, the operator $\Ccal_A$ is the Hessian of the polynomial $\det$ and is therefore self-adjoint.  Differentiating
\[
(\cof A)A^T=(\det A)I
\]
in the direction $H$ gives (iii).  Finally, $\cof(tA)=t^2\cof A$, and differentiating at $t=1$ gives (iv).
\end{proof}

\begin{lemma}[Cross-product covariance]\label{lem:cross-cof}
For every $T\in M_3(\mathbb R)$ and $u,v\in\mathbb R^3$,
\[
(Tu)\times(Tv)=(\cof T)(u\times v).
\]
\end{lemma}

\begin{proof}
For invertible $T$, pair both sides with $Tw$ and use the determinant formula for the scalar triple product together with $(\cof T)^TT=(\det T)I$.  Singular $T$ follow by polynomial continuity.
\end{proof}

\begin{lemma}[Inverse of the positive cofactor map]\label{lem:cof-inverse}
If $C\in\Sym_3^+$, then there is a unique $P\in\Sym_3^+$ with $C=\cof P$, namely
\[
P=\sqrt{\det C}\,C^{-1}.
\]
In particular $\det(\cof P)=(\det P)^2$.
\end{lemma}

\begin{proof}
For $P>0$, $\cof P=(\det P)P^{-1}$.  The stated inverse follows immediately.
\end{proof}

\section{A global matrix parametrization of left-invariant metrics}\label{sec:graph}

Represent $g_e$ relative to the background inner product by
\[
g_e(X,Y)=\ip{LX}{Y}_0,\qquad L\in\Sym_6^+.
\]
Thus the space of left-invariant metrics is $\Sym_6^+$.
Write $\mathbf E=(E_1,E_2,E_3)$ and $\mathbf F=(F_1,F_2,F_3)$
as row vectors of basis elements. Multiplication on the right by a
matrix denotes the corresponding linear combinations.

\begin{proposition}[Graph parametrization]\label{prop:graph}
Every $L\in\Sym_6^+$ is represented uniquely by
\[
(P,Q,M)\in\Sym_3^+\times\Sym_3^+\times M_3(\mathbb R)
\]
as follows.  Let
\[
D=(\cof P)^{1/2},\qquad E=(\cof Q)^{1/2},
\]
and define
\[
A(P,Q,M)=
\begin{pmatrix}
D&0\\
MD&E
\end{pmatrix}.
\]
Then
\[
L=(A^TA)^{-1}.
\]
Moreover the frame
\[
(\mathbf X,\mathbf Y)=(\mathbf E,\mathbf F)A^T
\]
is $g$-orthonormal.  The resulting map
\[
(P,Q,M)\longmapsto L
\]
is a smooth diffeomorphism from $\Sym_3^+\times\Sym_3^+\times M_3(\mathbb R)$ onto $\Sym_6^+$.
\end{proposition}

\begin{proof}
Put $K=L^{-1}$ and write
\[
K=\begin{pmatrix}K_{11}&K_{12}\\K_{21}&K_{22}\end{pmatrix}.
\]
Set
\[
E=K_{22}^{1/2},\qquad W=E^{-1}K_{21}.
\]
The matrix
\[
K_{11}-K_{12}K_{22}^{-1}K_{21}=K_{11}-W^TW
\]
is the Schur complement of the positive-definite block $K_{22}$ in $K$.  A standard block-matrix criterion for positive definiteness shows that this Schur complement is positive definite.  Hence its positive square root
\[
D=(K_{11}-W^TW)^{1/2}
\]
is defined and
\[
A:=\begin{pmatrix}D&0\\W&E\end{pmatrix}
\]
satisfies $A^TA=K$.  Lemma~\ref{lem:cof-inverse} gives unique $P,Q>0$ with $D^2=\cof P$ and $E^2=\cof Q$, and then $M=WD^{-1}$.  This reconstruction is unique.  All operations used in both directions---block extraction, inversion, Schur complementation, and the positive-definite square root---are smooth on their respective positive-definite domains.  Hence the parametrization and its inverse are smooth.  Finally $L=A^{-1}A^{-T}$ implies $ALA^T=I$, so the displayed frame is orthonormal.
\end{proof}

\begin{definition}[Graph coordinates and graph triple]\label{def:graph-triple}
For a left-invariant metric $g$, the unique triple $(P,Q,M)$ supplied by Proposition~\ref{prop:graph} is called the \emph{graph triple} of $g$, and $(P,Q,M)$ are called its \emph{graph coordinates}.  This terminology refers only to the block lower-triangular orthonormal-frame matrix
\[
A(P,Q,M)=\begin{pmatrix}D&0\\MD&E\end{pmatrix}.
\]
The matrix $M$ records the component of the second group of orthonormal frame vectors in the first factor and is called the \emph{mixing matrix}.
\end{definition}

\begin{proposition}[Volume]\label{prop:volume}
Let the background left-invariant metric be the one induced by $\ip{\cdot}{\cdot}_0$ on the two factors.  The ratio of the Riemannian volume density of $g$ to that background volume density (equivalently, the ratio of total volumes, since both are left-invariant) is
\[
V(g):=\frac{\operatorname{Vol}(G,g)}{\operatorname{Vol}(G,g_{\rm can})}
=\sqrt{\det L}=(\det P\det Q)^{-1}.
\]
In particular, the volume is independent of $M$.
\end{proposition}

\begin{proof}
Since $L=(A^TA)^{-1}$,
\[
\sqrt{\det L}=|\det A|^{-1}=(\det D\det E)^{-1}.
\]
Because $D^2=\cof P$ and $E^2=\cof Q$, Lemma~\ref{lem:cof-inverse} yields $\det D=\det P$ and $\det E=\det Q$.
\end{proof}

Under the bases fixed above, the adjoint group is identified with
\[
\Ad(G)=SO(3)\times SO(3).
\]
For $(U,V)$ in this group, the induced change of oriented orthonormal bases acts on graph coordinates by
\[
(P,Q,M)\longmapsto(UPU^T,VQV^T,VMU^T).
\]
This is the inner-automorphism action on left-invariant metrics, up to the convention of writing the pullback using the inverse group element.  In particular, applying this action replaces the metric by one obtained from it by an inner automorphism and therefore does not change any isometry-invariant statement.

\begin{proposition}[Inner-automorphism action and signed SVD]\label{prop:equivariance}
For $T=\diag(U,V)$,
\[
A(UPU^T,VQV^T,VMU^T)=TA(P,Q,M)T^T.
\]
Consequently scalar curvature and volume are invariant under the action above.  Moreover every $M$ can be placed, by this $SO(3)\times SO(3)$ action, in signed singular-value form
\[
M=\diag(m_1,m_2,m_3),\qquad m_i\in\mathbb R.
\]
Here ``singular-value form'' refers to the usual singular-value decomposition $M=V^T\Sigma U$ with orthogonal $U,V$ and diagonal $\Sigma\ge0$; the adjective ``signed'' records the modification needed when both changes of basis are required to preserve orientation.
If $\rank M=1$ or $2$, all nonzero $m_i$ may be chosen positive.
\end{proposition}

\begin{proof}
The positive square root commutes with orthogonal conjugation, which gives the matrix identity.  For the last statement begin with an ordinary $O(3)\times O(3)$ singular-value decomposition and absorb any orientation-reversing sign into one diagonal singular value.  If a singular value is zero, the sign can be absorbed in the zero direction.
\end{proof}

\begin{definition}[Signed-SVD gauge]\label{def:signed-svd-gauge}
A graph triple $(P,Q,M)$ is said to be in \emph{signed-SVD gauge} if the
$SO(3)\times SO(3)$ action above has been used to choose a representative
for which
\[
M=\diag(m_1,m_2,m_3),\qquad m_i\in\mathbb R.
\]
Here ``gauge'' means only a choice of oriented orthonormal bases in the two
factors.  The same change of bases sends $P$ to $UPU^T$ and $Q$ to
$VQV^T$, so it does not impose an additional geometric condition on the
metric.  Equivalently, we replace the metric by its pullback under an inner
automorphism; this conjugates $K(g)$ and therefore preserves the distinction
between trivial and nontrivial inner isotropy.

The numbers $|m_i|$ are the ordinary singular values of $M$.  The signs are
allowed because the two changes of basis are required to lie in $SO(3)$ rather
than merely in $O(3)$.  If $M$ is invertible, the product
$m_1m_2m_3$ has the same sign as $\det M$.  If $\rank M<3$, a zero singular
direction absorbs the orientation sign, so every nonzero $m_i$ may be chosen
positive, as stated in Proposition~\ref{prop:equivariance}.
\end{definition}

From now on, whenever equations involving vanishing off-diagonal entries or the factor interchange are used,
we choose signed-SVD gauge in the sense of
Definition~\ref{def:signed-svd-gauge}.

\section{Scalar curvature in graph coordinates}\label{sec:scalar}

The scalar curvature of a left-invariant metric is constant on $G$, because left translations act transitively by isometries.  We denote this constant by $S(P,Q,M)$ when the metric has graph triple $(P,Q,M)$.  Introduce the two polynomial expressions
\[
\Phi(H)=\frac12(\tr H)^2-\tr(H^2),
\qquad
\Lop(H)=\tr(H^2)I-H^2-2\cof H,
\]
and use the abbreviations
\[
C=\cof M,\qquad R=CP-QM.
\]

\begin{theorem}\label{thm:scalar}
For the left-invariant metric determined by $(P,Q,M)$,
\begin{equation}\label{eq:scalar}
S(P,Q,M)=
\Phi(P)+\Phi(Q)
-\frac12\tr\bigl(M^TM\Lop(P)\bigr)
-\frac12\norm{CP-QM}_{\Frob}^{2}.
\end{equation}
\end{theorem}

\begin{proof}
By Proposition~\ref{prop:equivariance}, the $SO(3)\times SO(3)$ action independently conjugates $P$ and $Q$, while replacing $M$ by $VMU^T$.  We may therefore choose $U$ and $V$ that diagonalize $P$ and $Q$, respectively; no simultaneous diagonalization of $M$ is being assumed.  Both sides of \eqref{eq:scalar} are invariant under this action, so it suffices to compute with
\[
P=\diag(p_1,p_2,p_3),\qquad Q=\diag(q_1,q_2,q_3),
\]
with $M$ still arbitrary.
For each cyclic permutation $(i,j,k)$ of $(1,2,3)$, let
\[
d_i=\sqrt{p_jp_k},\qquad e_i=\sqrt{q_jq_k}.
\]
so $D=\diag(d_i)$ and $E=\diag(e_i)$.  The graph orthonormal frame is
\[
X_i=d_iE_i,
\qquad
Y_i=\sum_{r=1}^3m_{ir}d_rE_r+e_iF_i.
\]
A Lie group is called \emph{unimodular} if $\tr(\operatorname{ad}_X)=0$ for every element $X$ of its Lie algebra, where $\operatorname{ad}_X(Y)=[X,Y]$.  Every compact Lie group is unimodular, so this applies to $G$.  For a unimodular Lie group and an orthonormal frame $(Z_a)$, the standard scalar-curvature formula (see, for example, \cite{Milnor}) is
\begin{equation}\label{eq:unimod-scalar}
S=-\frac14\sum_{a,b}\norm{[Z_a,Z_b]}_g^2
-\frac12\sum_a \Kill(Z_a,Z_a).
\end{equation}
Since $\Kill=-2\ip{\cdot}{\cdot}_0$, the Killing-form term in \eqref{eq:unimod-scalar} is
\[
\tr(\cof P)+\tr(M^TM\cof P)+\tr(\cof Q).
\]
Furthermore
\[
A^{-T}=\begin{pmatrix}D^{-1}&-M^TE^{-1}\\0&E^{-1}\end{pmatrix}.
\]
Using Lemma~\ref{lem:cross-cof} and $D^{-1}\cof D=P$, $E^{-1}\cof E=Q$, one obtains
\begin{align*}
\sum_{i,j}\norm{[X_i,X_j]}_g^2&=2\tr(P^2),\\
\sum_{i,j}\norm{[X_i,Y_j]}_g^2
&=\tr\Bigl(M^TM\bigl(\tr(P^2)I-P^2\bigr)\Bigr),\\
\sum_{i,j}\norm{[Y_i,Y_j]}_g^2
&=2\norm{R}_{\Frob}^2+2\tr(Q^2).
\end{align*}
Substitution into \eqref{eq:unimod-scalar} gives
\begin{align*}
S={}&\tr(\cof P)-\frac12\tr(P^2)
+\tr(\cof Q)-\frac12\tr(Q^2)\\
&+\tr(M^TM\cof P)
-\frac12\tr\Bigl(M^TM(\tr(P^2)I-P^2)\Bigr)
-\frac12\norm{R}_{\Frob}^2.
\end{align*}
Using
\[
\tr(\cof P)=\frac12\bigl((\tr P)^2-\tr(P^2)\bigr)
\]
and the analogous identity for $Q$ yields \eqref{eq:scalar}.
\end{proof}

\begin{remark}
After diagonalizing $P,Q$, formula \eqref{eq:scalar} expands term-by-term to the scalar-curvature polynomial used in \cite{BCHL}; see Appendix~\ref{app:dictionary}.
\end{remark}

\section{Einstein critical equations}\label{sec:einstein-crit}

We now explain precisely why the Einstein equation gives a finite-dimensional critical-point system.  For a compact Riemannian manifold, the first variation of the Einstein--Hilbert functional (see, for example, \cite{Besse})
\[
g\longmapsto\int_G \operatorname{Scal}_g\,dV_g
\]
shows that an Einstein metric is stationary under volume-preserving metric variations, meaning that the first variation vanishes for every such variation.  We only need the necessary consequence obtained by restricting those variations to left-invariant metrics.  Because a left-invariant metric has constant scalar curvature, the Einstein--Hilbert functional on a fixed-volume family differs from the scalar curvature $S$ only by a constant factor.  Therefore every left-invariant Einstein metric is a critical point of $S$ on the corresponding set $\{V=\text{constant}\}$.  We do not need the converse in the argument below.

By Proposition~\ref{prop:volume},
\[
V(P,Q,M)=(\det P\det Q)^{-1}.
\]
For symmetric variations $H$ of $P$ and $K$ of $Q$,
\[
DV[H,K,0]
=-V\tr(P^{-1}H)-V\tr(Q^{-1}K),
\qquad D_MV=0.
\]
The gradients $\nabla_PS$, $\nabla_QS$, and $\nabla_MS$ are taken with respect to the Frobenius inner product in the corresponding matrix spaces.  The level set of $V$ is regular, since, for example, the variation $H=P$, $K=0$ gives $DV[P,0,0]=-3V\neq0$.  The Lagrange-multiplier rule therefore gives a scalar $\kappa$ such that
\begin{equation}\label{eq:EL}
\nabla_PS=\kappa P^{-1},
\qquad
\nabla_QS=\kappa Q^{-1},
\qquad
\nabla_MS=0.
\end{equation}
The homogeneity $S(tP,tQ,M)=t^2S(P,Q,M)$ gives
\[
2S=\ip{\nabla_PS}{P}+\ip{\nabla_QS}{Q}=6\kappa,
\]
so
\begin{equation}\label{eq:kappa}
\kappa=\frac{S}{3}=2\lambda.
\end{equation}
The Einstein constant is positive.  Recall that a \emph{Killing field} is a vector field whose local flow consists of isometries.  For an Einstein metric, the standard Bochner identity for a Killing field $X$ gives
\[
\int_G\norm{\nabla X}^2=\lambda\int_G\norm{X}^2.
\]
Thus $\lambda<0$ is impossible.  If $\lambda=0$, every Killing field generated by the transitive left-translation action is parallel.  These fields span every tangent space, so the curvature tensor vanishes.  This would make the compact simply connected manifold $S^3\times S^3$ flat, which is impossible.  Therefore $\lambda>0$ and $\kappa>0$.

Set
\begin{equation}\label{eq:rho}
\rho_P=\frac{2\kappa}{\det P},
\qquad
\rho_Q=\frac{2\kappa}{\det Q}.
\end{equation}

\begin{proposition}[Complete gradient formulas]\label{prop:gradients}
Let $N=M^TM$ and $\sym Z=\tfrac12(Z+Z^T)$.  Then
\begin{align}
\nabla_PS={}&(\tr P)I-2P-(\tr N)P+\frac12(NP+PN)
+\Ccal_P(N)-\sym(C^TR),\label{eq:Pgrad}\\
\nabla_QS={}&(\tr Q)I-2Q+\frac12(RM^T+MR^T),\label{eq:Qgrad}\\
\nabla_MS={}&-M\Lop(P)-\Ccal_M(RP)+QR.\label{eq:Mgrad}
\end{align}
\end{proposition}

\begin{proof}
For a symmetric variation $H$,
\[
D\Phi_P[H]=\ip{(\tr P)I-2P}{H},
\]
and
\[
D\Lop_P[H]
=2\tr(PH)I-(PH+HP)-2\Ccal_P(H).
\]
Using self-adjointness of $\Ccal_P$ gives
\begin{align*}
-\frac12\tr\bigl(ND\Lop_P[H]\bigr)
={}&-(\tr N)\tr(PH)
+\frac12\tr\bigl(N(PH+HP)\bigr)
+\ip{\Ccal_P(N)}{H}.
\end{align*}
Since $D_PR[H]=CH$, differentiating the last squared-norm term with respect to $P$ gives
$-\ip{R}{CH}=\ip{-\sym(C^TR)}{H}$, proving \eqref{eq:Pgrad}.  Similarly $D_QR[K]=-KM$ gives \eqref{eq:Qgrad}.  Finally, for an arbitrary variation $H$ of $M$,
\[
D_MR[H]=\Ccal_M(H)P-QH,
\]
so
\begin{align*}
-\ip{R}{D_MR[H]}
&=-\ip{RP}{\Ccal_M(H)}+\ip{QR}{H}\\
&=\ip{-\Ccal_M(RP)+QR}{H},
\end{align*}
which yields \eqref{eq:Mgrad}.
\end{proof}

If $t>0$, the coordinate change $(P,Q,M)\mapsto(tP,tQ,M)$ multiplies the matrix $A$ of Proposition~\ref{prop:graph} by $t$ and therefore replaces the metric by the homothetic metric $t^{-2}g$.  Under this homothety one has $\kappa\mapsto t^2\kappa$ and
\[
\rho_P\mapsto\rho_P/t,
\qquad
\rho_Q\mapsto\rho_Q/t.
\]
Thus, whenever $\rho_P>0$, we may choose the overall scale of the metric so that
\begin{equation}\label{eq:rhoPnorm}
\rho_P=1.
\end{equation}

\section{A radial inequality and the first eigenvalue bound}\label{sec:first-bound}

In this and the next several sections, a \emph{strict spectral upper bound} of the form $H\prec4I$ is simply a matrix-order statement in the sense fixed in Section~\ref{sec:linear}; it says that every eigenvalue of the symmetric matrix $H$ is $<4$.

\begin{lemma}[Spectrum of $\Lop(P)$]\label{lem:L-spectrum}
If the eigenvalues of $P>0$ are $p_1,p_2,p_3$, then in an eigenbasis
\[
\Lop(P)=\diag\bigl((p_2-p_3)^2,(p_3-p_1)^2,(p_1-p_2)^2\bigr).
\]
Hence $\Lop(P)\succeq0$, with equality when $P$ is scalar, and
$\rank\Lop(P)\in\{0,2,3\}$.
\end{lemma}

\begin{proof}
Orthogonally diagonalize $P$.  The operations $P\mapsto P^2$, $P\mapsto\cof P$, and $P\mapsto\tr(P^2)I$ are equivariant under orthogonal conjugation, so it is enough to substitute $P=\diag(p_1,p_2,p_3)$ into the definition of $\Lop$.  This gives the displayed diagonal matrix.  Its entries are nonnegative.  They all vanish exactly when $p_1=p_2=p_3$; if exactly two eigenvalues agree, precisely one displayed entry vanishes, while if the eigenvalues are pairwise distinct none vanishes.  Thus the only possible ranks are $0,2,3$.
\end{proof}

Define the nonnegative scalar
\begin{equation}\label{eq:ell}
\ell=\tr(M^TM\Lop(P))
=\norm{M\Lop(P)^{1/2}}_{\Frob}^2\ge0.
\end{equation}
The \emph{zero-$\ell$ locus} is the set $\{\ell=0\}$.
We call
\begin{equation}\label{eq:full-rank-positive-ell-region}
\det M\neq0,\qquad \ell>0
\end{equation}
the \emph{full-rank, positive-$\ell$ region}.

Pair the $M$-equation with $M$, corresponding to the radial variation $M\mapsto(1+t)M$. Set
\[
A_0=CP,
\qquad B_0=QM,
\qquad R=A_0-B_0,
\]
\[
x=\norm{A_0}_{\Frob},
\qquad y=\norm{B_0}_{\Frob},
\qquad c=\ip{A_0}{B_0},
\]
and
\[
\alpha=\ip{R}{A_0}=x^2-c,
\qquad
\beta=\ip{R}{B_0}=c-y^2.
\]

\begin{proposition}[Radial identity and strict estimate]\label{prop:radial}
If $\nabla_MS=0$, then
\begin{equation}\label{eq:radial}
\ell+2\alpha-\beta=0.
\end{equation}
If in addition $\rank M\ge2$ and $\ell>0$, then
\begin{equation}\label{eq:alpha-strict}
\alpha<-\ell.
\end{equation}
\end{proposition}

\begin{proof}
Pair \eqref{eq:Mgrad} with $M$.  Lemma~\ref{lem:cof-basic} gives
\[
\ip{\Ccal_M(RP)}{M}
=\ip{RP}{\Ccal_M(M)}
=2\ip{RP}{C}=2\alpha,
\]
while $\ip{QR}{M}=\beta$.  This proves \eqref{eq:radial}.  Equivalently,
\begin{equation}\label{eq:cidentity}
3c=2x^2+y^2+\ell.
\end{equation}
If $\rank M\ge2$, then $C=\cof M\neq0$ and hence $x>0$.  Cauchy--Schwarz yields
\[
2x^2+y^2+\ell=3c\le3xy,
\]
so
\begin{equation}\label{eq:xyproduct}
\ell\le(y-x)(2x-y).
\end{equation}
Since $\ell>0$, this forces $x<y<2x$.  Moreover $x+y>2(2x-y)$, and therefore
\[
y^2-x^2=(y-x)(x+y)>2(y-x)(2x-y)\ge2\ell.
\]
Using \eqref{eq:cidentity},
\[
\alpha=x^2-c=\frac{x^2-y^2-\ell}{3}<-\ell.
\]
\end{proof}

Define the polynomial matrix map
\begin{equation}\label{eq:Kop}
\Kop(H)=\cof H-2\tr(H)I+4H.
\end{equation}
For symmetric $H$, the three-dimensional Cayley--Hamilton identity gives
\[
\cof H=H^2-(\tr H)H+\frac12\bigl((\tr H)^2-\tr(H^2)\bigr)I.
\]
Thus $\Kop(H)$ is a matrix polynomial in $H$ and hence commutes with $H$.

\begin{proposition}[Cofactor identity]\label{prop:master}
At a graph triple satisfying the Einstein equations \eqref{eq:EL}, let
\[
\widehat Q=\rho_QQ.
\]
Then
\begin{equation}\label{eq:KQ}
\Kop(\widehat Q)=\rho_Q(RM^T+MR^T),
\end{equation}
and
\begin{equation}\label{eq:master}
\frac1{\rho_Q}Q\Kop(\widehat Q)
=M\Lop(P)M^T+\alpha I-RR^T.
\end{equation}
\end{proposition}

\begin{proof}
From \eqref{eq:Qgrad} and the Einstein equation,
\[
2\tr(Q)I-4Q+RM^T+MR^T=2\kappa Q^{-1}.
\]
Since $2\kappa Q^{-1}=\rho_Q\cof Q$, multiplying by $\rho_Q$ gives \eqref{eq:KQ}.  On the other hand, \eqref{eq:Mgrad} gives
\[
QR=M\Lop(P)+\Ccal_M(RP).
\]
Right-multiplying by $M^T$ and using Lemma~\ref{lem:cof-basic},
\begin{align*}
QRM^T
&=M\Lop(P)M^T+\Ccal_M(RP)M^T\\
&=M\Lop(P)M^T+\ip{C}{RP}I-CPR^T\\
&=M\Lop(P)M^T+\alpha I-CPR^T.
\end{align*}
Adding $QMR^T$ and using $QM-CP=-R$ gives
\[
Q(RM^T+MR^T)=M\Lop(P)M^T+\alpha I-RR^T.
\]
Combining with \eqref{eq:KQ} proves \eqref{eq:master}.
\end{proof}

\begin{corollary}\label{cor:K-negative}
At a graph triple satisfying the Einstein equations \eqref{eq:EL}, set
$\widehat Q=\rho_QQ$.  If $\rank M\ge2$ and $\ell>0$, then
\[
\Kop(\widehat Q)\prec0.
\]
\end{corollary}

\begin{proof}
Put $H_0=M\Lop(P)M^T\succeq0$.  Since $\tr H_0=\ell$, one has $\lambda_{\max}(H_0)\le\ell$.  Proposition~\ref{prop:radial} gives $H_0+\alpha I\prec0$, and subtracting $RR^T\succeq0$ preserves strict negativity.  Thus $Q\Kop(\widehat Q)\prec0$.  The matrices $Q$ and $\Kop(\widehat Q)$ commute, because $\widehat Q$ is a positive scalar multiple of $Q$ and $\Kop$ is polynomial in its argument.  Since $Q\succ0$, simultaneous diagonalization then shows that $Q\Kop(\widehat Q)\prec0$ is equivalent to $\Kop(\widehat Q)\prec0$.
\end{proof}

\begin{lemma}[The threshold $4$]\label{lem:four}
If $H\in\Sym_3^+$ and $\Kop(H)\prec0$, then
\[
H\prec4I.
\]
\end{lemma}

\begin{proof}
Let $0<r\le s\le t$ be the eigenvalues of $H$.  The eigenvalue of $\Kop(H)$ in the $t$-eigendirection is
\[
k_t=rs-2r-2s+2t.
\]
Assume $t\ge4$.  If $s\ge2$, then
\[
k_t=r(s-2)+2(t-s)>0.
\]
If $s<2$, then $r\le s<2$ and
\[
k_t\ge rs-2r-2s+8=(2-r)(2-s)+4>0.
\]
Both contradict $\Kop(H)\prec0$.
\end{proof}

\begin{corollary}[First strict spectral upper bound]\label{cor:first-bound}
At a graph triple satisfying the Einstein equations \eqref{eq:EL}, if $\rank M\ge2$ and $\ell>0$, then
\begin{equation}\label{eq:first-bound}
\rho_QQ\prec4I.
\end{equation}
\end{corollary}

\begin{proof}
Corollary~\ref{cor:K-negative} gives $\Kop(\rho_QQ)\prec0$, and Lemma~\ref{lem:four} applies because $\rho_QQ$ is positive definite.
\end{proof}

\section{Lower-rank cases and the zero-\texorpdfstring{$\ell$}{ell} locus}\label{sec:low-rank}

\begin{lemma}\label{lem:rank0}
If $M=0$, then $K(g)\neq\{e\}$.
\end{lemma}

\begin{proof}
Choose a unit eigenvector $u$ of the symmetric matrix $P$.  Its orthogonal complement $u^\perp$ is $P$-invariant, so the rotation $U\in SO(3)$ by $\pi$ about the axis $\mathbb Ru$ commutes with $P$.  Put $V=I$.  Then $UPU^T=P$, $VQV^T=Q$, and $VMU^T=M=0$.  Proposition~\ref{prop:equivariance} therefore gives a nontrivial inner isometry.
\end{proof}

\begin{lemma}\label{lem:rank1}
There is no left-invariant Einstein metric $g$ whose graph triple $(P,Q,M)$ satisfies $\rank M=1$.
\end{lemma}

\begin{proof}
If $\rank M=1$, then $C=\cof M=0$ and $R=-QM$.
Consequently $\alpha=0$ and $\beta=-\norm{QM}_{\Frob}^2$.
The radial identity in Proposition~\ref{prop:radial} becomes
\[
0=\ell+2\alpha-\beta=\ell+\norm{QM}_{\Frob}^2.
\]
This is impossible: $\ell\ge0$, while $QM\neq0$ because $Q$ is
invertible and $M\neq0$.
\end{proof}

\begin{lemma}\label{lem:rank2}
There is no left-invariant Einstein metric $g$ whose graph triple $(P,Q,M)$ satisfies $\rank M=2$.
\end{lemma}

\begin{proof}
If $\ell>0$, Corollary~\ref{cor:K-negative} gives $\Kop(\widehat Q)\prec0$.  For $0\neq v\in\ker M^T$, equation \eqref{eq:KQ} gives
\[
v^T\Kop(\widehat Q)v=0,
\]
a contradiction.  Thus a hypothetical left-invariant Einstein metric with $\rank M=2$ must have $\ell=0$.  Since
\[
M\Lop(P)^{1/2}=0,
\]
we have $\operatorname{im}\Lop(P)^{1/2}\subseteq\ker M$, hence $\rank\Lop(P)\le1$.  Lemma~\ref{lem:L-spectrum} forces $\Lop(P)=0$, so $P=pI$.

Now put
\[
M=\diag(m_1,m_2,0),\qquad m_1,m_2>0.
\]
Then
\[
C=\diag(0,0,m_1m_2),\qquad R=pC-QM.
\]
Since $\Lop(P)=0$,
\[
(\nabla_MS)_{11}=-\Ccal_M(pR)_{11}+(QR)_{11}.
\]
Here
\[
\Ccal_M(pR)_{11}=m_1m_2^2p^2,
\qquad
(QR)_{11}=-m_1(Q^2)_{11},
\]
so
\[
(\nabla_MS)_{11}
=-m_1\bigl(m_2^2p^2+Q_{11}^2+Q_{12}^2+Q_{13}^2\bigr)<0,
\]
again contradicting \eqref{eq:EL}.
\end{proof}

\begin{lemma}\label{lem:full-ell0}
Let $g$ be a left-invariant Einstein metric with graph triple $(P,Q,M)$.  If $M$ is invertible and $\ell=0$, then $K(g)\neq\{e\}$.
\end{lemma}

\begin{proof}
Since $M^TM\succ0$, $\Lop(P)\succeq0$, and
$\tr(M^TM\Lop(P))=0$, we have $\Lop(P)=0$, hence $P=pI$.
Choose signed-SVD gauge $M=\diag(m_1,m_2,m_3)$, where every $m_i$
is nonzero. The off-diagonal entries of \eqref{eq:Pgrad}, evaluated
at $P=pI$, are
\[
(\nabla_PS)_{jk}=\frac12m_i(m_j^2+m_k^2)Q_{jk},
\qquad (i,j,k)\text{ cyclic}.
\]
Indeed, all terms in \eqref{eq:Pgrad} except
$-\sym(C^TR)$ are diagonal at $P=pI$; using
$C=\cof M$ and $R=pC-QM$ gives the displayed formula.
The right-hand side $\kappa P^{-1}$ of the Einstein equation is
diagonal, so $Q_{jk}=0$ for every $j\neq k$.
Thus $P,Q,M$ are simultaneously diagonal. A simultaneous coordinate
$\pi$-rotation in the two factors fixes all three matrices and yields
a nonidentity inner isometry by Proposition~\ref{prop:equivariance}.
\end{proof}

Consequently, the hypothesis $K(g)=\{e\}$ forces $\rank M=3$ and $\ell>0$.

\section{Interchanging the factors and the second strict spectral upper bound}\label{sec:swap}

On the full-rank, positive-$\ell$ region impose the normalization \eqref{eq:rhoPnorm}.  In signed-SVD gauge $M=M^T$.  Put
\[
\Bmix=I+M^2,
\qquad
d=\sqrt{\det \Bmix}.
\]
Let $\tau:G\to G$ interchange the two factors and set $g^\vee=\tau^*g$.  In the reference basis this interchange is represented by
\[
J=\begin{pmatrix}0&I\\I&0\end{pmatrix}.
\]
We will use the polar decomposition with the positive factor on the right: for an invertible real matrix $T$, this is the factorization $T=R^TH$ with $R$ orthogonal and $H=(T^TT)^{1/2}$ positive definite.

\begin{proposition}[Factor-interchange formulas]\label{prop:swap}
Set
\[
D_0=\Bmix^{-1/2}E,
\qquad
E_0=\Bmix^{1/2}D,
\]
and take polar decompositions with the positive factors on the right
\[
D_0=R_1^TH_1,
\qquad
E_0=R_2^TH_2,
\qquad
H_1,H_2\in\Sym_3^+.
\]
In the present situation the determinants are positive, so $R_1,R_2\in SO(3)$.  The unique graph triple of $g^\vee$ given by Proposition~\ref{prop:graph} is
\begin{equation}\label{eq:swap}
P^\vee=d^{-1}Q^{1/2}\Bmix Q^{1/2},
\qquad
Q^\vee=dP^{1/2}\Bmix^{-1}P^{1/2},
\qquad
M^\vee=R_2MR_1^T.
\end{equation}
Moreover
\begin{equation}\label{eq:swapdet}
\det P^\vee=d^{-1}\det Q,
\qquad
\det Q^\vee=d\det P.
\end{equation}
\end{proposition}

\begin{proof}
Write
\[
A=\begin{pmatrix}D&0\\MD&E\end{pmatrix},
\qquad
D=(\cof P)^{1/2},\quad E=(\cof Q)^{1/2}.
\]
Define
\[
O=
\begin{pmatrix}
-\Bmix^{-1/2}M&\Bmix^{-1/2}\\
\Bmix^{-1/2}&\Bmix^{-1/2}M
\end{pmatrix}.
\]
Since $M$ commutes with $\Bmix=I+M^2$, direct block multiplication gives $OO^T=I$.  Hence left multiplication by $O$ merely changes the orthonormal frame for the swapped metric, and
\[
OAJ=
\begin{pmatrix}
D_0&0\\
MD_0&E_0
\end{pmatrix}.
\]
The blocks $D_0,E_0$ need not be symmetric, so one must perform the polar decompositions above.  Since
\[
H_1^2=D_0^TD_0=E\Bmix^{-1}E,
\qquad
H_2^2=E_0^TE_0=D\Bmix D,
\]
a further left multiplication by $\diag(R_1,R_2)$ produces the graph form of Proposition~\ref{prop:graph} and gives
\[
M^\vee=R_2MR_1^T,
\qquad
\cof P^\vee=E\Bmix^{-1}E,
\qquad
\cof Q^\vee=D\Bmix D.
\]
Now
\[
E=\sqrt{\det Q}\,Q^{-1/2},
\qquad
D=\sqrt{\det P}\,P^{-1/2}.
\]
Applying Lemma~\ref{lem:cof-inverse} gives \eqref{eq:swap}; determinants give \eqref{eq:swapdet}.
\end{proof}

Since $g$ and $g^\vee$ are isometric, $\kappa$ is unchanged.  Under $\rho_P=1$ we have $2\kappa=\det P$, and therefore
\[
\rho_{Q^\vee}=\frac{2\kappa}{\det Q^\vee}=\frac1d.
\]
Thus
\begin{equation}\label{eq:swappedQhat}
\rho_{Q^\vee}Q^\vee=P^{1/2}\Bmix^{-1}P^{1/2}.
\end{equation}
The matrices $P^{1/2}\Bmix^{-1}P^{1/2}$ and $\Bmix^{-1/2}P\Bmix^{-1/2}$ have the same positive eigenvalues.

\begin{proposition}[Second strict spectral upper bound]\label{prop:second-bound}
Let $g$ be a left-invariant Einstein metric whose graph triple $(P,Q,M)$ has $\rank M=3$.  If $K(g)=\{e\}$, then after normalizing so that $\rho_P=1$,
\begin{equation}\label{eq:second-bound}
\Bmix^{-1/2}P\Bmix^{-1/2}\prec4I.
\end{equation}
\end{proposition}

\begin{proof}
The factor interchange conjugates inner automorphisms by $(U,V)\mapsto(V,U)$, hence $K(g)=\{e\}$ if and only if $K(g^\vee)=\{e\}$.  Let
\[
\ell^\vee=\tr\bigl((M^\vee)^TM^\vee\Lop(P^\vee)\bigr).
\]
If $\ell^\vee=0$, the invertibility of $M^\vee$ implies $\Lop(P^\vee)=0$, so $P^\vee$ is scalar.  After putting $M^\vee$ in signed-SVD gauge, the same off-diagonal argument as in Lemma~\ref{lem:full-ell0} makes $Q^\vee$ diagonal.  A coordinate $\pi$-rotation is then a nontrivial inner symmetry, contradicting $K(g^\vee)=\{e\}$.  Hence $\ell^\vee>0$.  Applying the first strict spectral upper bound \eqref{eq:first-bound} to the swapped triple and using \eqref{eq:swappedQhat} gives
\[
P^{1/2}\Bmix^{-1}P^{1/2}\prec4I,
\]
which is equivalent to \eqref{eq:second-bound}.
\end{proof}

\section{Off-diagonal Einstein equations}\label{sec:offdiag}

In the full-rank signed-SVD gauge write
\[
M=\diag(m_1,m_2,m_3),
\qquad
m_i\neq0,
\qquad
a_i=m_i^2>0,
\qquad
b_i=1+a_i.
\]
For every cyclic permutation $(i,j,k)$ of $(1,2,3)$, set
\[
S_i=a_j+a_k>0,
\qquad
\Delta_i=S_i+4.
\]
We use the following \emph{edge-coordinate} convention for the off-diagonal entries:
\[
P=\begin{pmatrix}
p_1&x_3&x_2\\
x_3&p_2&x_1\\
x_2&x_1&p_3
\end{pmatrix},
\qquad
Q=\begin{pmatrix}
q_1&y_3&y_2\\
y_3&q_2&y_1\\
y_2&y_1&q_3
\end{pmatrix}.
\]
Thus, for cyclic $(i,j,k)$,
\[
x_i=P_{jk}=P_{kj},
\qquad
y_i=Q_{jk}=Q_{kj}.
\]
The purpose of this section is to extract from the scalar-curvature formula the six equations obtained by varying these six symmetric off-diagonal coordinates.

\begin{proposition}[Off-diagonal part of scalar curvature]\label{prop:Soff}
The part of scalar curvature that depends on the variables $x_i,y_i$ is
\begin{equation}\label{eq:Soff}
S_{\rm off}
=-\frac12\sum_{i=1}^3
\left[b_i\Delta_i x_i^2-2m_iS_ix_iy_i+\Delta_i y_i^2\right].
\end{equation}
In particular, no monomial involving two distinct edge indices occurs.
\end{proposition}

\begin{proof}
We expand separately the three terms in the scalar-curvature formula
\[
S=\Phi(P)+\Phi(Q)
-\frac12\tr\bigl(M^TM\Lop(P)\bigr)
-\frac12\norm{(\cof M)P-QM}_{\Frob}^2.
\]

\smallskip
\noindent\emph{1. The $\Phi(P)$ and $\Phi(Q)$ terms.}
Since
\[
\tr(P)=p_1+p_2+p_3
\]
is independent of $x_1,x_2,x_3$, whereas
\[
\tr(P^2)=p_1^2+p_2^2+p_3^2+2(x_1^2+x_2^2+x_3^2),
\]
the $x$-dependent part of
$\Phi(P)=\tfrac12(\tr P)^2-\tr(P^2)$ is
\[
-2\sum_{i=1}^3x_i^2.
\]
The same calculation for $Q$ gives the $y$-dependent term
\[
-2\sum_{i=1}^3y_i^2.
\]

\smallskip
\noindent\emph{2. The term containing $\Lop(P)$.}
For cyclic $(i,j,k)$, direct substitution into
\[
\Lop(P)=\tr(P^2)I-P^2-2\cof P
\]
gives
\begin{equation}\label{eq:Ldiag-expanded}
\Lop(P)_{ii}
=(p_j-p_k)^2+4x_i^2+x_j^2+x_k^2.
\end{equation}
For example,
\begin{align*}
\Lop(P)_{11}
&=\tr(P^2)-(P^2)_{11}-2(\cof P)_{11}\\
&=\bigl(p_1^2+p_2^2+p_3^2+2x_1^2+2x_2^2+2x_3^2\bigr)
  -\bigl(p_1^2+x_2^2+x_3^2\bigr)\\
&\qquad -2(p_2p_3-x_1^2)\\
&=(p_2-p_3)^2+4x_1^2+x_2^2+x_3^2,
\end{align*}
and the other two identities follow by cyclic permutation.

Because
\[
M^TM=\diag(a_1,a_2,a_3),
\]
only the diagonal entries of $\Lop(P)$ enter the trace.  From
\eqref{eq:Ldiag-expanded}, the coefficient of $x_i^2$ in
$\tr(M^TM\Lop(P))$ is
\[
4a_i+a_j+a_k=4a_i+S_i.
\]
Consequently, the complete off-diagonal part of this term is
\begin{equation}\label{eq:Lterm-off}
-\frac12\sum_{i=1}^3(4a_i+S_i)x_i^2.
\end{equation}

\smallskip
\noindent\emph{3. The squared norm of $R$.}
Put
\[
C=\cof M=\diag(m_2m_3,m_3m_1,m_1m_2),
\qquad
R=CP-QM.
\]
For cyclic $(i,j,k)$,
\begin{align*}
R_{jk}
&=C_{jj}P_{jk}-Q_{jk}M_{kk}\\
&=(m_km_i)x_i-y_im_k
 =m_k(m_ix_i-y_i),\\
R_{kj}
&=C_{kk}P_{kj}-Q_{kj}M_{jj}\\
&=(m_im_j)x_i-y_im_j
 =m_j(m_ix_i-y_i).
\end{align*}
Thus the two entries indexed by the unordered pair $\{j,k\}$ satisfy
\[
R_{jk}^2+R_{kj}^2
=(m_j^2+m_k^2)(m_ix_i-y_i)^2
=S_i(m_ix_i-y_i)^2.
\]
The diagonal entries of $R$ are independent of all $x_r,y_r$, and the three unordered pairs $\{2,3\}$, $\{3,1\}$, $\{1,2\}$ are disjoint.  Hence the complete off-diagonal part of this squared norm is
\begin{equation}\label{eq:Rterm-off}
-\frac12\sum_{i=1}^3S_i(m_ix_i-y_i)^2.
\end{equation}

Combining the preceding formulas, the summand associated with a fixed index $i$ is
\begin{align*}
&-2x_i^2-2y_i^2
-\frac12(4a_i+S_i)x_i^2
-\frac12S_i(m_ix_i-y_i)^2\\
&\quad=-\frac12\Bigl[
\bigl(4+4a_i+S_i+a_iS_i\bigr)x_i^2
-2m_iS_ix_iy_i
+(4+S_i)y_i^2
\Bigr].
\end{align*}
Finally,
\[
4+4a_i+S_i+a_iS_i
=(1+a_i)(4+S_i)=b_i\Delta_i,
\qquad
4+S_i=\Delta_i,
\]
which proves \eqref{eq:Soff}.
\end{proof}

We next convert the derivatives of \eqref{eq:Soff} into the Einstein equations.  Because $x_i$ represents the two equal entries $P_{jk}=P_{kj}$, a variation of $x_i$ is the symmetric matrix
\[
H_i=E_{jk}+E_{kj}.
\]
Therefore
\[
\frac{\partial S}{\partial x_i}
=DS_P[H_i]
=2(\nabla_PS)_{jk}.
\]
Using $\nabla_PS=\kappa P^{-1}$ and
$P^{-1}=(\det P)^{-1}(\cof P)^T$, with $P$ and $\cof P$ symmetric, we obtain
\begin{equation}\label{eq:x-derivative-einstein}
\frac{\partial S}{\partial x_i}
=\frac{2\kappa}{\det P}(\cof P)_{jk}
=\rho_P(\cof P)_{jk}.
\end{equation}
The identical argument for $y_i$ gives
\begin{equation}\label{eq:y-derivative-einstein}
\frac{\partial S}{\partial y_i}
=\rho_Q(\cof Q)_{jk}.
\end{equation}

\begin{corollary}[Six off-diagonal Einstein equations]\label{cor:offeq}
Let $(P,Q,M)$ be an Einstein graph triple in the full-rank signed-SVD gauge fixed above.  For every cyclic $(i,j,k)$,
\begin{align}
-b_i\Delta_i x_i+m_iS_i y_i
&=\rho_P(\cof P)_{jk},\label{eq:offP-general}\\
m_iS_i x_i-\Delta_i y_i
&=\rho_Q(\cof Q)_{jk}.\label{eq:offQ-general}
\end{align}
Under the normalization $\rho_P=1$, put
\[
\zeta=\rho_Q>0.
\]
Then
\begin{align}
-b_i\Delta_i x_i+m_iS_i y_i
&=(\cof P)_{jk},\label{eq:offP}\\
m_iS_i x_i-\Delta_i y_i
&=\zeta(\cof Q)_{jk}.\label{eq:offQ}
\end{align}
Moreover,
\begin{equation}\label{eq:cofedge}
(\cof P)_{jk}=x_jx_k-p_ix_i,
\qquad
(\cof Q)_{jk}=y_jy_k-q_iy_i.
\end{equation}
\end{corollary}

\begin{proof}
Differentiating \eqref{eq:Soff} gives
\[
\frac{\partial S}{\partial x_i}
=-b_i\Delta_i x_i+m_iS_i y_i,
\qquad
\frac{\partial S}{\partial y_i}
=m_iS_i x_i-\Delta_i y_i.
\]
Equations \eqref{eq:x-derivative-einstein} and
\eqref{eq:y-derivative-einstein} now give
\eqref{eq:offP-general}--\eqref{eq:offQ-general}.

For completeness, take $(i,j,k)=(1,2,3)$.  The $(2,3)$ cofactor of $P$ is
\[
(\cof P)_{23}
=-\det\begin{pmatrix}p_1&x_3\\x_2&x_1\end{pmatrix}
=x_2x_3-p_1x_1.
\]
Cyclic permutation gives the first identity in \eqref{eq:cofedge}; the second is identical.
\end{proof}

Since $M$ is invertible, $m_i\neq0$, and since $S_i=a_j+a_k>0$, the equations above give the following consequences for the vanishing pattern of the off-diagonal entries:
\begin{align}
x_i=0
&\Longrightarrow
m_iS_i y_i=x_jx_k,\label{eq:supportP}\\
y_i=0
&\Longrightarrow
m_iS_i x_i=\zeta y_jy_k.\label{eq:supportQ}
\end{align}
These two implications will be used repeatedly below.

\section{Common coordinate axes and vanishing patterns}\label{sec:support}

We first make precise how a common coordinate axis produces an inner isometry.

\begin{lemma}[Criterion for an inner symmetry]\label{lem:stabilizer}
Suppose $(U,V)\in SO(3)\times SO(3)$ satisfies
\[
UPU^T=P,
\qquad
VQV^T=Q,
\qquad
VMU^T=M.
\]
Then the inner automorphism whose differential is
$T=\diag(U,V)$ is an isometry of the left-invariant metric determined by $(P,Q,M)$.
\end{lemma}

\begin{proof}
The relation $UPU^T=P$ implies
\[
U(\cof P)U^T=\cof P.
\]
By uniqueness of the positive square root,
\[
UDU^T=D,
\qquad
D=(\cof P)^{1/2}.
\]
Similarly,
\[
VEV^T=E,
\qquad
E=(\cof Q)^{1/2}.
\]
For
\[
A=\begin{pmatrix}D&0\\MD&E\end{pmatrix},
\qquad
T=\diag(U,V),
\]
we compute
\begin{align*}
TAT^T
&=\begin{pmatrix}
UDU^T&0\\
VMDU^T&VEV^T
\end{pmatrix}\\
&=\begin{pmatrix}
D&0\\
(VMU^T)(UDU^T)&E
\end{pmatrix}\\
&=\begin{pmatrix}D&0\\MD&E\end{pmatrix}=A.
\end{align*}
It follows that
\[
T(A^TA)T^T=A^TA.
\]
Since $L=(A^TA)^{-1}$, the same identity holds for $L$:
\[
TLT^T=L.
\]
Thus the inner automorphism with differential $T$ preserves the metric at the identity and hence, by left invariance, preserves the metric globally.
\end{proof}

For $i\in\{1,2,3\}$, let
\[
R_i=2e_ie_i^T-I.
\]
This is the rotation by $\pi$ about the axis $\mathbb Re_i$; in coordinates,
\[
R_1=\diag(1,-1,-1),
\quad
R_2=\diag(-1,1,-1),
\quad
R_3=\diag(-1,-1,1).
\]
Because $P$ is symmetric, the axis $\mathbb Re_i$ is $P$-invariant exactly when
\begin{equation}\label{eq:P-axis-condition}
x_j=x_k=0,
\end{equation}
where $(i,j,k)$ is cyclic.  In that case $R_iPR_i^T=P$.  Likewise,
\begin{equation}\label{eq:Q-axis-condition}
y_j=y_k=0
\end{equation}
means that $\mathbb Re_i$ is $Q$-invariant and $R_iQR_i^T=Q$.
Since $M$ is diagonal,
\[
R_iMR_i^T=M.
\]
Consequently, if \eqref{eq:P-axis-condition} and
\eqref{eq:Q-axis-condition} hold simultaneously, Lemma~\ref{lem:stabilizer} applied to $(U,V)=(R_i,R_i)$ produces a nonidentity element of $K(g)$.

Under the contradiction hypothesis $K(g)=\{e\}$, no coordinate axis can therefore be invariant for both $P$ and $Q$.

\begin{definition}[Edge support, support type, and missing index]\label{def:support-type}
In signed-SVD coordinates, the \emph{edge support} of $P$ is
\[
I_P=\{i\in\{1,2,3\}:x_i\neq0\},
\]
and the edge support of $Q$ is
\[
I_Q=\{i\in\{1,2,3\}:y_i\neq0\}.
\]
The ordered pair $(|I_P|,|I_Q|)$ is called the \emph{support type}.  If, for example, $|I_P|=2$, the unique index in $\{1,2,3\}\setminus I_P$ is called the \emph{missing $P$-edge index}; the analogous terminology is used for $Q$.  These are coordinate notions tied to the signed-SVD gauge and are not claimed to be invariants of an arbitrary matrix presentation.
\end{definition}

\begin{lemma}[Exhaustion of support types]\label{lem:support-class}
Let $(P,Q,M)$ be the graph triple of a left-invariant Einstein metric,
in signed-SVD gauge with $M$ invertible and normalized by $\rho_P=1$.
Suppose that $P,Q$ have no common invariant coordinate axis. Then
\[
(|I_P|,|I_Q|)\in\{(3,3),(2,3),(3,2),(2,2)\}.
\]
In the $(2,2)$ case, the missing index of $I_P$ is different from the missing index of $I_Q$.
\end{lemma}

\begin{proof}
We examine the four possible values of $|I_P|$.

\smallskip
\noindent\emph{Case $|I_P|=0$.}
Then $x_1=x_2=x_3=0$.  Equation \eqref{eq:supportP} gives
\[
m_iS_i y_i=0
\]
for every $i$.  Because $m_iS_i\neq0$, all $y_i$ vanish.  Hence every coordinate axis is invariant for both $P$ and $Q$, contrary to the hypothesis.

\smallskip
\noindent\emph{Case $|I_P|=1$.}
Suppose $I_P=\{i\}$.  Then $x_i\neq0$ and
$x_j=x_k=0$.  Applying \eqref{eq:supportP} at the two missing indices gives
\[
m_jS_jy_j=x_kx_i=0,
\qquad
m_kS_ky_k=x_ix_j=0.
\]
Thus $y_j=y_k=0$.  Equations
\eqref{eq:P-axis-condition} and \eqref{eq:Q-axis-condition} show that $\mathbb Re_i$ is a common invariant axis, again a contradiction.  Notice also that \eqref{eq:supportQ} at index $i$ forces $y_i\neq0$, so in fact $I_Q=\{i\}$.

\smallskip
\noindent\emph{Case $|I_P|=2$.}
Let $i$ be the unique missing index, so
\[
x_i=0,
\qquad
x_jx_k\neq0.
\]
Equation \eqref{eq:supportP} at index $i$ gives
\begin{equation}\label{eq:missing-x-forces-y}
y_i=\frac{x_jx_k}{m_iS_i}\neq0.
\end{equation}
Therefore $I_Q$ is nonempty.  It cannot have cardinality one: if
$I_Q=\{i\}$, then $y_j=y_k=0$, and \eqref{eq:supportQ} at index $j$ gives
\[
m_jS_jx_j=\zeta y_ky_i=0,
\]
contrary to $x_j\neq0$.  Hence $|I_Q|\in\{2,3\}$.
If $|I_Q|=2$, its missing index cannot be $i$, because
\eqref{eq:missing-x-forces-y} shows that $y_i\neq0$.
Thus this case yields either $(2,3)$ or $(2,2)$ with different missing indices.

\smallskip
\noindent\emph{Case $|I_P|=3$.}
All $x_i$ are nonzero.  If $I_Q=\varnothing$, then
\eqref{eq:supportQ} gives $m_iS_ix_i=0$ for every $i$, impossible.
If $|I_Q|=1$, say $I_Q=\{i\}$, then $y_j=y_k=0$, and
\eqref{eq:supportQ} at index $j$ again gives
\[
m_jS_jx_j=\zeta y_ky_i=0,
\]
contrary to $x_j\neq0$.  Therefore $|I_Q|\in\{2,3\}$, which gives the types $(3,2)$ and $(3,3)$.

These cases exhaust all possibilities.
\end{proof}

\begin{definition}[Fully mixed symmetric matrix]\label{def:fully-mixed-matrix}
After a signed-SVD coordinate system has been chosen, a symmetric
$3\times3$ matrix $H$ is called \emph{fully mixed} if all three of its
off-diagonal entries in that coordinate system are nonzero:
\begin{equation}\label{eq:fully-mixed-matrix}
H_{12}H_{13}H_{23}\neq0.
\end{equation}
This is deliberately a coordinate-dependent shorthand, not an invariant
property of an abstract symmetric bilinear form.  In the notation of
Section~\ref{sec:offdiag}, $P$ is fully mixed iff
$x_1x_2x_3\neq0$, and $Q$ is fully mixed iff
$y_1y_2y_3\neq0$.
\end{definition}

\section{A rank-one decomposition for fully mixed matrices}\label{sec:rankone}

The next elementary lemma converts the cofactor entries of a fully mixed
symmetric matrix, in the sense of Definition~\ref{def:fully-mixed-matrix},
into diagonal parameters.  The strict inequality $H\prec4I$ then forces at
least two of those parameters below $4$.

\begin{lemma}[Fully mixed decomposition]\label{lem:rankone}
Let $H\in\Sym_3$ have all three off-diagonal entries nonzero.  For cyclic $(i,j,k)$ write
\[
h_i=H_{jk}.
\]
Then there exist
\[
D_H=\diag(d_1,d_2,d_3),
\qquad
\varepsilon\in\{\pm1\},
\qquad
v=(v_1,v_2,v_3)^T,
\qquad
v_1v_2v_3\neq0,
\]
such that
\begin{equation}\label{eq:rankone-decomp}
H=D_H+\varepsilon vv^T.
\end{equation}
If in addition $H\prec4I$, then at least two of the three numbers $d_i$ are strictly smaller than $4$.  Finally,
\begin{equation}\label{eq:cof-rankone}
(\cof H)_{jk}=-d_i h_i
\end{equation}
for every cyclic $(i,j,k)$.
\end{lemma}

\begin{proof}
Set
\[
\varepsilon=\operatorname{sgn}(H_{12}H_{13}H_{23}).
\]
Because all three factors are nonzero,
\[
\varepsilon\frac{H_{12}H_{13}}{H_{23}}>0.
\]
Choose
\[
v_1=\sqrt{\varepsilon\frac{H_{12}H_{13}}{H_{23}}}>0,
\qquad
v_2=\varepsilon\frac{H_{12}}{v_1},
\qquad
v_3=\varepsilon\frac{H_{13}}{v_1}.
\]
Then
\[
\varepsilon v_1v_2=H_{12},
\qquad
\varepsilon v_1v_3=H_{13}.
\]
Moreover,
\begin{align*}
\varepsilon v_2v_3
&=\varepsilon
\left(\varepsilon\frac{H_{12}}{v_1}\right)
\left(\varepsilon\frac{H_{13}}{v_1}\right)\\
&=\frac{\varepsilon H_{12}H_{13}}{v_1^2}
=H_{23}.
\end{align*}
Thus the off-diagonal entries of $\varepsilon vv^T$ agree with those of $H$.  Defining
\[
d_i=H_{ii}-\varepsilon v_i^2
\]
proves \eqref{eq:rankone-decomp}.

We now impose $H\prec4I$.
If $\varepsilon=1$, then for each coordinate vector $e_i$,
\[
H_{ii}=d_i+v_i^2<4,
\]
so
\[
d_i<4-v_i^2<4
\]
for all three indices.

Suppose $\varepsilon=-1$.  If two indices $r\neq s$ satisfied
$d_r,d_s\ge4$, take
\[
u=v_se_r-v_re_s.
\]
The vector $u$ is nonzero because $v_rv_s\neq0$, is supported in the $(r,s)$-coordinate plane, and satisfies $u^Tv=0$.  Since now
$H=D_H-vv^T$,
\begin{align*}
u^THu
&=u^TD_Hu-(u^Tv)^2\\
&=d_rv_s^2+d_sv_r^2\\
&\ge4(v_s^2+v_r^2)=4\norm{u}^2.
\end{align*}
This contradicts $H\prec4I$, which requires
$u^THu<4\norm{u}^2$ for every nonzero $u$.  Hence at most one $d_i$ can be at least $4$, and at least two are strictly below $4$.

Finally, for cyclic $(i,j,k)$, direct expansion of the $(j,k)$ cofactor gives
\begin{equation}\label{eq:cof-direct-rankone}
(\cof H)_{jk}=h_jh_k-H_{ii}h_i.
\end{equation}
From \eqref{eq:rankone-decomp},
\[
h_i=\varepsilon v_jv_k,
\qquad
h_j=\varepsilon v_kv_i,
\qquad
h_k=\varepsilon v_iv_j.
\]
Therefore
\[
h_jh_k
=v_i^2v_jv_k
=\varepsilon v_i^2h_i.
\]
Since $H_{ii}=d_i+\varepsilon v_i^2$, substitution into
\eqref{eq:cof-direct-rankone} yields
\[
(\cof H)_{jk}
=\varepsilon v_i^2h_i-(d_i+\varepsilon v_i^2)h_i
=-d_ih_i.
\]
\end{proof}

In a decomposition $H=D_H+\varepsilon vv^T$ as in Lemma~\ref{lem:rankone}, we call the entries $d_i$ of the diagonal matrix $D_H$ the \emph{diagonal remainder parameters}.  

\section{Exclusion of all full-rank support types}\label{sec:fullrank-exclusion}

Throughout this section, suppose for contradiction that
\[
K(g)=\{e\}.
\]
The preceding reductions place us in the full-rank, positive-$\ell$ region defined in \eqref{eq:full-rank-positive-ell-region}, and both strict spectral upper bounds hold:
\begin{equation}\label{eq:two-strict-bounds}
\widehat Q:=\zeta Q\prec4I,
\qquad
\widehat P:=\Bmix^{-1/2}P\Bmix^{-1/2}\prec4I,
\qquad
\Bmix=\diag(b_1,b_2,b_3).
\end{equation}
Two elementary consequences will be used repeatedly.  First, evaluating
$4I-\widehat Q$ and $4I-\widehat P$ on the coordinate vectors gives
\begin{equation}\label{eq:diagonal-strict-bounds}
\zeta q_i<4,
\qquad
\frac{p_i}{b_i}<4
\qquad(i=1,2,3).
\end{equation}
Second, multiplication by the positive diagonal matrices $\zeta I$ and
$\Bmix^{-1/2}$ does not change which off-diagonal entries vanish:
\[
(\widehat Q)_{jk}=\zeta y_i,
\qquad
(\widehat P)_{jk}=\frac{x_i}{\sqrt{b_jb_k}}.
\]
Thus $Q$ is fully mixed iff $\widehat Q$ is fully mixed, and $P$ is fully mixed iff $\widehat P$ is fully mixed.

We record explicitly how the off-diagonal Einstein equations interact with the decomposition of Lemma~\ref{lem:rankone}.

\begin{lemma}[Edge ratios and diagonal remainders]\label{lem:edge-ratios}
Under the standing Einstein and normalization hypotheses of this section,
let $(i,j,k)$ be cyclic.

\begin{enumerate}[label=\textup{(\roman*)}]
\item Suppose $P$ is fully mixed and write
\[
P=\diag(c_1,c_2,c_3)+\varepsilon zz^T
\]
as in Lemma~\ref{lem:rankone}.  Since $x_i\neq0$, define
\[
\Theta_i=\frac{y_i}{m_ix_i}.
\]
Then
\begin{equation}\label{eq:ci}
c_i=b_i\Delta_i-a_iS_i\Theta_i,
\end{equation}
and hence
\begin{equation}\label{eq:ci-threshold}
\frac{c_i}{b_i}-4
=\frac{S_i}{b_i}(b_i-a_i\Theta_i).
\end{equation}

\item Suppose $Q$ is fully mixed and $x_iy_i\neq0$.  Define
\[
\Theta_i=\frac{y_i}{m_ix_i}\neq0
\]
and write
\[
\widehat Q=\diag(e_1,e_2,e_3)+\eta ww^T.
\]
Then
\begin{equation}\label{eq:ei}
e_i=\Delta_i-\frac{S_i}{\Theta_i},
\end{equation}
and therefore
\begin{equation}\label{eq:ei-threshold}
e_i-4=S_i\frac{\Theta_i-1}{\Theta_i}.
\end{equation}
\end{enumerate}
\end{lemma}

\begin{proof}
For part (i), Lemma~\ref{lem:rankone} gives
\[
(\cof P)_{jk}=-c_ix_i.
\]
Substitution into \eqref{eq:offP} gives
\[
-b_i\Delta_i x_i+m_iS_i y_i=-c_ix_i.
\]
Divide by $x_i\neq0$ and use
\[
\frac{m_iy_i}{x_i}=m_i^2\Theta_i=a_i\Theta_i
\]
to obtain \eqref{eq:ci}.  Since $\Delta_i=S_i+4$,
\begin{align*}
\frac{c_i}{b_i}-4
&=\Delta_i-\frac{a_iS_i}{b_i}\Theta_i-4\\
&=S_i-\frac{a_iS_i}{b_i}\Theta_i
=\frac{S_i}{b_i}(b_i-a_i\Theta_i),
\end{align*}
which is \eqref{eq:ci-threshold}.

For part (ii), Lemma~\ref{lem:rankone}, applied to $\widehat Q$, gives
\[
(\cof\widehat Q)_{jk}=-e_i(\widehat Q)_{jk}=-\zeta e_i y_i.
\]
Since $\cof(\zeta Q)=\zeta^2\cof Q$, this identity is equivalent to
\[
\zeta(\cof Q)_{jk}=-e_iy_i.
\]
Substituting into \eqref{eq:offQ} yields
\[
m_iS_ix_i-\Delta_i y_i=-e_iy_i.
\]
After division by $y_i\neq0$,
\[
e_i=\Delta_i-\frac{m_iS_ix_i}{y_i}
=\Delta_i-\frac{S_i}{\Theta_i},
\]
which is \eqref{eq:ei}.  Subtracting $4$ gives
\eqref{eq:ei-threshold}.
\end{proof}

\subsection{The fully mixed type \texorpdfstring{$(3,3)$}{(3,3)}}

Assume
\[
x_1x_2x_3y_1y_2y_3\neq0.
\]
Both $P$ and $\widehat Q$ are fully mixed.  Write
\[
P=\diag(c_1,c_2,c_3)+\varepsilon zz^T,
\qquad
\widehat Q=\diag(e_1,e_2,e_3)+\eta ww^T,
\]
and define
\[
\Theta_i=\frac{y_i}{m_ix_i}\neq0
\qquad(i=1,2,3).
\]

Because $\widehat Q\prec4I$, Lemma~\ref{lem:rankone} says that
$e_i<4$ for at least two indices.  By \eqref{eq:ei-threshold},
\[
e_i<4
\quad\Longleftrightarrow\quad
\frac{\Theta_i-1}{\Theta_i}<0
\quad\Longleftrightarrow\quad
0<\Theta_i<1.
\]
Hence the set
\begin{equation}\label{eq:index-set-small}
J_-:=\{i:0<\Theta_i<1\}
\end{equation}
has cardinality at least two.

Next apply $\Bmix^{-1/2}$ to the decomposition of $P$:
\begin{align*}
\widehat P
&=\Bmix^{-1/2}P\Bmix^{-1/2}\\
&=\diag\left(\frac{c_1}{b_1},
              \frac{c_2}{b_2},
              \frac{c_3}{b_3}\right)
 +\varepsilon\widetilde z\widetilde z^T,
\qquad
\widetilde z=\Bmix^{-1/2}z.
\end{align*}
The matrix $\widehat P$ is fully mixed and satisfies
$\widehat P\prec4I$, so Lemma~\ref{lem:rankone} implies
$c_i/b_i<4$ for at least two indices.  Formula
\eqref{eq:ci-threshold} and the positivity of $S_i,b_i,a_i$ give
\begin{align*}
\frac{c_i}{b_i}<4
&\quad\Longleftrightarrow\quad
b_i-a_i\Theta_i<0\\
&\quad\Longleftrightarrow\quad
\Theta_i>\frac{b_i}{a_i}
=1+\frac1{a_i}>1.
\end{align*}
Thus the set
\begin{equation}\label{eq:index-set-large}
J_+:=\left\{i:\Theta_i>\frac{b_i}{a_i}\right\}
\end{equation}
also has cardinality at least two.

Any two subsets of a three-element set with cardinality at least two have nonempty intersection.  Choose
$i\in J_-\cap J_+$.  Then simultaneously
\[
0<\Theta_i<1
\qquad\text{and}\qquad
\Theta_i>1+\frac1{a_i}>1,
\]
which is impossible.  Therefore the support type $(3,3)$ cannot occur.

\subsection{The type \texorpdfstring{$(2,3)$}{(2,3)}}

Assume that $i$ is the missing $P$-edge:
\begin{equation}\label{eq:23-support}
x_i=0,
\qquad
x_jx_k\neq0,
\qquad
y_1y_2y_3\neq0.
\end{equation}
For $r=j$, the cofactor formula \eqref{eq:cofedge} gives
\[
(\cof P)_{ki}=x_kx_i-p_jx_j=-p_jx_j.
\]
For $r=k$ it gives similarly
\[
(\cof P)_{ij}=x_ix_j-p_kx_k=-p_kx_k.
\]
Thus, for each $r\in\{j,k\}$, equation \eqref{eq:offP} becomes
\begin{equation}\label{eq:23-edge-equation}
-b_r\Delta_rx_r+m_rS_ry_r=-p_rx_r,
\end{equation}
or equivalently
\begin{equation}\label{eq:23-rearranged}
m_rS_ry_r=(b_r\Delta_r-p_r)x_r.
\end{equation}
All factors $m_r,x_r,y_r$ are nonzero, so define
\[
\Theta_r=\frac{y_r}{m_rx_r}.
\]
Dividing \eqref{eq:23-rearranged} by $x_r$ gives
\begin{equation}\label{eq:23-theta-exact}
a_rS_r\Theta_r=b_r\Delta_r-p_r.
\end{equation}
The diagonal estimate $p_r/b_r<4$ in
\eqref{eq:diagonal-strict-bounds} implies
\[
b_r\Delta_r-p_r
>b_r(\Delta_r-4)
=b_rS_r.
\]
Since $a_rS_r>0$, division yields
\begin{equation}\label{eq:23-theta-large}
\Theta_r>\frac{b_r}{a_r}>1
\qquad(r=j,k).
\end{equation}

Because $Q$ is fully mixed, write
\[
\widehat Q=\diag(e_1,e_2,e_3)+\eta ww^T.
\]
For $r=j,k$, Lemma~\ref{lem:edge-ratios}(ii) applies.  Using
\eqref{eq:23-theta-large},
\[
e_r-4
=S_r\frac{\Theta_r-1}{\Theta_r}>0.
\]
Hence
\[
e_j>4,
\qquad
e_k>4.
\]
At most the remaining number $e_i$ can be below $4$, whereas
Lemma~\ref{lem:rankone} applied to
$\widehat Q\prec4I$ requires at least two of the $e_r$ to be below $4$.  This contradiction excludes the type $(2,3)$.

\subsection{The type \texorpdfstring{$(3,2)$}{(3,2)}}

Assume that $i$ is the missing $Q$-edge:
\begin{equation}\label{eq:32-support}
y_i=0,
\qquad
y_jy_k\neq0,
\qquad
x_1x_2x_3\neq0.
\end{equation}
For $r=j$, formula \eqref{eq:cofedge} gives
\[
(\cof Q)_{ki}=y_ky_i-q_jy_j=-q_jy_j,
\]
and similarly
\[
(\cof Q)_{ij}=-q_ky_k.
\]
Thus equation \eqref{eq:offQ} becomes, for $r=j,k$,
\begin{equation}\label{eq:32-edge-equation}
m_rS_rx_r-\Delta_ry_r=-\zeta q_ry_r,
\end{equation}
or
\begin{equation}\label{eq:32-rearranged}
m_rS_rx_r=(\Delta_r-\zeta q_r)y_r.
\end{equation}
Define
\[
\Theta_r=\frac{y_r}{m_rx_r}
\qquad(r=j,k).
\]
After division by $S_ry_r$, equation
\eqref{eq:32-rearranged} gives
\begin{equation}\label{eq:32-theta-inverse}
\frac1{\Theta_r}
=\frac{m_rx_r}{y_r}
=\frac{\Delta_r-\zeta q_r}{S_r}.
\end{equation}
The diagonal estimate $\zeta q_r<4$ implies
\[
\Delta_r-\zeta q_r>\Delta_r-4=S_r,
\]
so \eqref{eq:32-theta-inverse} yields
\begin{equation}\label{eq:32-theta-small}
0<\Theta_r<1
\qquad(r=j,k).
\end{equation}

Now $P$ is fully mixed.  Write
\[
P=\diag(c_1,c_2,c_3)+\varepsilon zz^T.
\]
Because $m_ix_i\neq0$ and $y_i=0$, we may also set
\[
\Theta_i=\frac{y_i}{m_ix_i}=0.
\]
Formula \eqref{eq:ci-threshold} applies to all three indices.  At the missing $Q$-edge,
\[
\frac{c_i}{b_i}-4
=\frac{S_i}{b_i}b_i=S_i>0.
\]
For $r=j,k$, using \eqref{eq:32-theta-small},
\[
b_r-a_r\Theta_r
=1+a_r(1-\Theta_r)>0,
\]
and therefore
\[
\frac{c_r}{b_r}-4
=\frac{S_r}{b_r}(b_r-a_r\Theta_r)>0.
\]
We have proved
\[
\frac{c_1}{b_1}>4,
\qquad
\frac{c_2}{b_2}>4,
\qquad
\frac{c_3}{b_3}>4.
\]
On the other hand,
\[
\widehat P
=\diag\left(\frac{c_1}{b_1},
              \frac{c_2}{b_2},
              \frac{c_3}{b_3}\right)
 +\varepsilon\widetilde z\widetilde z^T
\prec4I
\]
is fully mixed, so Lemma~\ref{lem:rankone} requires at least two of the diagonal remainders $c_r/b_r$ to be smaller than $4$.  This contradiction excludes the type $(3,2)$.

\subsection{The type \texorpdfstring{$(2,2)$}{(2,2)}}

By Lemma~\ref{lem:support-class}, the missing indices are different.  Let
\begin{equation}\label{eq:22-support}
x_i=0,
\qquad
y_j=0,
\qquad
i\neq j,
\end{equation}
and let $k$ be the remaining index.  Since $|I_P|=2$,
\[
x_j\neq0,
\qquad
x_k\neq0.
\]
For the $j$th edge equation, the cofactor formula gives
\[
(\cof P)_{ki}=x_kx_i-p_jx_j=-p_jx_j.
\]
Because $y_j=0$, equation \eqref{eq:offP} at index $j$ is therefore
\[
-b_j\Delta_jx_j=-p_jx_j.
\]
Division by $x_j\neq0$ gives the identity
\[
\frac{p_j}{b_j}=\Delta_j=S_j+4.
\]
Since $S_j=a_i+a_k>0$,
\[
\frac{p_j}{b_j}>4.
\]
This contradicts the diagonal estimate $p_j/b_j<4$ in
\eqref{eq:diagonal-strict-bounds}.  Hence the type $(2,2)$ is impossible.

All four support types allowed by Lemma~\ref{lem:support-class} have now been excluded.  Therefore no full-rank left-invariant Einstein metric with $\ell>0$ can have trivial inner isotropy.

\section{Proof of automatic symmetry}\label{sec:auto-proof}

\begin{proof}[Proof of Theorem~\ref{thm:auto}]
Let $(P,Q,M)$ be the graph triple of a left-invariant Einstein metric
$g$, and suppose that $K(g)=\{e\}$.
Lemma~\ref{lem:rank0} excludes $M=0$, while
Lemmas~\ref{lem:rank1} and \ref{lem:rank2} exclude ranks one and two.
Thus $M$ is invertible. Lemma~\ref{lem:full-ell0} also excludes
$\ell=0$, so $\ell>0$.

Choose signed-SVD gauge and normalize by $\rho_P=1$.
Corollary~\ref{cor:first-bound} and Proposition~\ref{prop:second-bound}
give both strict bounds \eqref{eq:two-strict-bounds}.
A common invariant coordinate axis of $P,Q$ would yield a nonidentity
inner isometry by Lemma~\ref{lem:stabilizer}. Therefore there is no
such axis, and Lemma~\ref{lem:support-class} restricts the support
type to $(3,3)$, $(2,3)$, $(3,2)$, or $(2,2)$, with distinct
missing indices in the last case.
Section~\ref{sec:fullrank-exclusion} excludes each of these four
types using these bounds and the off-diagonal Einstein
equations. This contradiction proves $K(g)\neq\{e\}$.
\end{proof}

\section{The remaining trace \texorpdfstring{$-2$}{-2} involution case: setup}\label{sec:z2-setup}

We now prove Theorem~\ref{thm:z2upgrade}, using the normal-form
coordinates of \cite[Section~3.2]{BCHL}. The theorem is invariant
under homothety, so in this part we impose unit \emph{relative} volume:
\[
V(g)=\frac{\operatorname{Vol}(G,g)}{\operatorname{Vol}(G,g_{\rm can})}=1.
\]
This normalization replaces the condition $\rho_P=1$ used in the
proof of Theorem~\ref{thm:auto}. The two normalizations are not imposed
simultaneously.

Assume that $g$ is preserved by an inner involution $\sigma$ with
$\tr\sigma=-2$. After an inner change of oriented bases in the two
factors, the normal form in \cite[Section~3.2]{BCHL} gives the
following $g$-orthonormal basis:
\begin{equation}\label{eq:z2-normal}
\begin{aligned}
(X_1,X_2,X_3,Y_1,Y_2,Y_3)
={}&(a_0E_1,b_0E_2,c_0E_3,\\[-0.25em]
&x_0E_1+d_0F_1,
 y_0E_2+w_0E_3+e_0F_2,
 \gamma_0E_2+z_0E_3+f_0F_3),
\end{aligned}
\end{equation}
where $a_0,\ldots,f_0>0$.  After the change of variables
\begin{equation}\label{eq:z2-coordinates}
\begin{gathered}
a_0=\sqrt{BC},\quad b_0=\sqrt{AC},\quad c_0=\sqrt{AB},
\qquad
d_0=\sqrt{EF},\quad e_0=\sqrt{DF},\quad f_0=\sqrt{DE},\\
x_0=X\sqrt{BC},\quad y_0=Y\sqrt{AC},\quad z_0=Z\sqrt{AB},
\quad w_0=W\sqrt{AB},\quad \gamma_0=\Th\sqrt{AC},
\end{gathered}
\end{equation}
we have
\[
A,B,C,D,E,F>0,
\qquad X,Y,Z,W,\Th\in\mathbb R,
\qquad ABCDEF=1.
\]
The last equation is the unit-relative-volume normalization, since $V(g)=(ABCDEF)^{-1}$.  The symbols $X,Y,Z,W,\Th$ are mixing parameters; they are real and need not be positive.  The symbol $\Th$ in this part is unrelated to the edge ratios $\Theta_i$ used in Sections~\ref{sec:fullrank-exclusion}--\ref{sec:auto-proof}.  Belgun--Cort\'es--Haupt--Lindemann computed the scalar curvature and the corresponding polynomial Einstein equations in these coordinates.  We reorganize those equations into a form adapted to the remaining cases.

Throughout this part, a \emph{branch} is a subset of the real solution set specified by the indicated parameter conditions.

Define
\begin{equation}\label{eq:z2-basic-defs}
\Delta:=YZ-W\Th,
\qquad
\Pi:=AD+BE+CF,
\qquad
\Omega:=AD+BF+CE,
\end{equation}
and
\begin{equation}\label{eq:z2-r-defs}
\begin{aligned}
\alpha_0&:=(B-C)^2+D^2,\\
r_1&:=(A-C)^2+E^2+C^2X^2,
& s_1&:=(A-B)^2+F^2+B^2X^2,\\
r_2&:=(A-B)^2+E^2+B^2X^2,
& s_2&:=(A-C)^2+F^2+C^2X^2.
\end{aligned}
\end{equation}
All four $r_i,s_i$ are strictly positive.  Put
\begin{equation}\label{eq:z2-S0}
S_0=-\frac12(A^2+B^2+C^2+D^2+E^2+F^2)
+AB+AC+BC+DE+DF+EF.
\end{equation}

\begin{lemma}[Paired form of scalar curvature]\label{lem:z2-regroup}
The scalar-curvature polynomial of \cite{BCHL} is equivalent to
\begin{equation}\label{eq:z2-S}
\begin{aligned}
S={}&S_0-\frac12\Bigl(
\alpha_0X^2+r_1Y^2+s_1Z^2+r_2W^2+s_2\Th^2+A^2\Delta^2
\Bigr)\\
&\quad+\Pi XYZ-\Omega XW\Th.
\end{aligned}
\end{equation}
\end{lemma}

\begin{proof}
Expand
\[
-\frac12A^2(YZ-W\Th)^2
=-\frac12A^2Y^2Z^2+A^2WYZ\Th-\frac12A^2W^2\Th^2
\]
and collect the quadratic terms in $X,Y,Z,W,\Th$.  This recovers the scalar-curvature expression of \cite[Eq. (3.17)]{BCHL} term by term.
\end{proof}

For a function $S=S(A,B,C,D,E,F,X,Y,Z,W,\Th)$, subscripts such as $S_A$ and $S_X$ denote ordinary partial derivatives.  On the normalized-volume hypersurface $ABCDEF=1$, the Lagrange-multiplier equations for an Einstein metric are
\begin{equation}\label{eq:z2-diag-einstein}
AS_A=BS_B=CS_C=DS_D=ES_E=FS_F=-\mu,
\qquad
S_X=S_Y=S_Z=S_W=S_{\Th}=0.
\end{equation}
Define
\begin{equation}\label{eq:z2-uv}
u:=X\Pi-A^2\Delta,
\qquad
v:=-X\Omega+A^2\Delta.
\end{equation}
Differentiating \eqref{eq:z2-S} gives four equations that split into two coupled pairs, one involving $(Y,Z)$ and one involving $(W,\Th)$:
\begin{equation}\label{eq:z2-pair-eqs}
r_1Y=uZ,
\qquad s_1Z=uY,
\qquad r_2W=v\Th,
\qquad s_2\Th=vW,
\end{equation}
and the $X$-equation
\begin{equation}\label{eq:z2-X-eq}
X\Hc=\Pi YZ-\Omega W\Th,
\end{equation}
where
\begin{equation}\label{eq:z2-H}
\Hc=(B-C)^2+D^2+B^2(W^2+Z^2)+C^2(Y^2+\Th^2)>0.
\end{equation}

Equation \eqref{eq:z2-pair-eqs} immediately implies
\begin{equation}\label{eq:z2-pair-zero}
Y=0\Longleftrightarrow Z=0,
\qquad
W=0\Longleftrightarrow\Th=0.
\end{equation}

\begin{definition}[Fully mixed anisotropic BCHL branch]
In the BCHL coordinates used in this part of the paper, the two \emph{paired mixing sectors} are the parameter pairs $(Y,Z)$ and $(W,\Th)$.  A solution is called
\emph{fully mixed} if both sectors are active:
\[
YZ\neq0,\qquad W\Th\neq0.
\]
By \eqref{eq:z2-pair-zero}, this is equivalent to requiring each of
$Y,Z,W,\Th$ to be nonzero.  It is called \emph{anisotropic} if, in addition,
\[
B\neq C,\qquad E\neq F.
\]
Thus the fully mixed anisotropic BCHL branch is precisely
\begin{equation}\label{eq:z2-genuine}
(YZ)(W\Th)(B-C)(E-F)\neq0.
\end{equation}
This terminology is separate from
Definition~\ref{def:fully-mixed-matrix}: here it describes two active mixing
pairs in the residual $\mathbb Z_2$ normal form, rather than the three
off-diagonal entries of a symmetric $3\times3$ matrix.
\end{definition}

We first exclude the fully mixed anisotropic branch. The complementary parameter cases are treated in Section~\ref{sec:z2-upgrade}.

\section{Sign structure and necessary inequalities}\label{sec:z2-sign}

Let
\begin{equation}\label{eq:z2-RV}
R_0:=\sqrt{r_1s_1}>0,
\qquad
V_0:=\sqrt{r_2s_2}>0.
\end{equation}
On the branch \eqref{eq:z2-genuine}, equations \eqref{eq:z2-pair-eqs} imply
\begin{equation}\label{eq:z2-u2v2}
u^2=R_0^2,
\qquad
v^2=V_0^2,
\end{equation}
and
\begin{equation}\label{eq:z2-sign-products}
\operatorname{sgn}u=\operatorname{sgn}(YZ),
\qquad
\operatorname{sgn}v=\operatorname{sgn}(W\Th).
\end{equation}

\begin{lemma}\label{lem:z2-strict-RV}
On \eqref{eq:z2-genuine},
\begin{equation}\label{eq:z2-strict-RV}
R_0>|X|(BE+CF),
\qquad
V_0>|X|(BF+CE).
\end{equation}
\end{lemma}

\begin{proof}
First,
\begin{align*}
R_0^2
&=\bigl((A-C)^2+E^2+C^2X^2\bigr)
  \bigl((A-B)^2+F^2+B^2X^2\bigr)\\
&>(E^2+C^2X^2)(F^2+B^2X^2),
\end{align*}
where strictness follows because $B\neq C$, so $A-B$ and $A-C$ cannot both vanish.  Moreover
\[
(E^2+C^2X^2)(F^2+B^2X^2)-X^2(BE+CF)^2
=(EF-BCX^2)^2\ge0.
\]
This proves the first inequality; the second is analogous.
\end{proof}

From \eqref{eq:z2-uv},
\begin{equation}\label{eq:z2-uplusv}
u+v=X(\Pi-\Omega)=X(B-C)(E-F).
\end{equation}
If $u$ and $v$ had the same sign, Lemma~\ref{lem:z2-strict-RV} would give
\begin{align*}
|u+v|=R_0+V_0
&>|X|(BE+CF+BF+CE)\\
&=|X|(B+C)(E+F)\\
&\ge |X|\,|B-C|\,|E-F|=|u+v|,
\end{align*}
a contradiction.  Thus
\begin{equation}\label{eq:z2-opposite}
uv<0.
\end{equation}
If $X=0$, then $u=-A^2\Delta$ and $v=A^2\Delta$, so $YZ$ and $W\Th$ have opposite signs.  But \eqref{eq:z2-X-eq} becomes
$\Pi YZ-\Omega W\Th=0$, impossible because $\Pi,\Omega>0$.  Hence
\begin{equation}\label{eq:z2-Xnonzero}
X\neq0.
\end{equation}
The signs in \eqref{eq:z2-X-eq} now give
\[
\operatorname{sgn}X=\operatorname{sgn}(YZ)=-\operatorname{sgn}(W\Th).
\]
Set
\begin{equation}\label{eq:z2-sign-variables}
x:=|X|>0,
\qquad
\varepsilon:=\operatorname{sgn}X,
\qquad
\rho:=\varepsilon YZ>0,
\qquad
\omega:=-\varepsilon W\Th>0.
\end{equation}
Then
\begin{equation}\label{eq:z2-sign-final}
u=\varepsilon R_0,
\qquad
v=-\varepsilon V_0,
\qquad
\Delta=\varepsilon(\rho+\omega).
\end{equation}

Introduce
\begin{equation}\label{eq:z2-sbte}
s=B+C,
\qquad b=B-C,
\qquad t=E+F,
\qquad e=E-F.
\end{equation}
Thus $s,t>0$ and $be\neq0$.  Equations \eqref{eq:z2-uplusv} and \eqref{eq:z2-sign-final} give
\begin{equation}\label{eq:z2-RminusV}
R_0-V_0=xbe.
\end{equation}
Direct expansion of \eqref{eq:z2-r-defs} gives
\begin{equation}\label{eq:z2-R2minusV2}
r_1s_1-r_2s_2
=bet\,[s(1+x^2)-2A].
\end{equation}
Using $R_0^2-V_0^2=(R_0-V_0)(R_0+V_0)$ and dividing by $be\neq0$ yields
\begin{equation}\label{eq:z2-RplusV}
R_0+V_0=\frac{t}{x}[s(1+x^2)-2A].
\end{equation}
Lemma~\ref{lem:z2-strict-RV} also gives $R_0+V_0>xst$, hence
\begin{equation}\label{eq:z2-s-lower}
s>2A.
\end{equation}
Put
\begin{equation}\label{eq:z2-mdef}
m=s-2A>0.
\end{equation}
Adding and subtracting \eqref{eq:z2-RminusV} and \eqref{eq:z2-RplusV} gives the following formulas
\begin{equation}\label{eq:z2-R-exact}
R_0=x(BE+CF)+\frac{tm}{2x},
\qquad
V_0=x(BF+CE)+\frac{tm}{2x}.
\end{equation}

Let
\begin{equation}\label{eq:z2-delta}
\delta=|\Delta|=\rho+\omega>0.
\end{equation}
The first identity in \eqref{eq:z2-uv}, together with \eqref{eq:z2-R-exact}, gives
\begin{equation}\label{eq:z2-delta-exact}
\delta=\frac{Dx}{A}-\frac{tm}{2A^2x}.
\end{equation}
In particular
\begin{equation}\label{eq:z2-D-lower}
D>\frac{tm}{2Ax^2}.
\end{equation}
On the other hand, differentiating \eqref{eq:z2-S} gives
\[
S_D=-D+t-DX^2+AX\Delta.
\]
Since $X\Delta=x\delta$, substitution of \eqref{eq:z2-delta-exact} yields
\begin{equation}\label{eq:z2-SD}
S_D=-D+\frac{t(2A-m)}{2A}.
\end{equation}
At a solution of the Einstein equations, $DS_D=-\mu$.  By the degree-two homogeneity in $A,\ldots,F$,
\[
2S=AS_A+BS_B+CS_C+DS_D+ES_E+FS_F=-6\mu,
\]
so $-\mu=S/3>0$.  Therefore $S_D>0$, and \eqref{eq:z2-SD} gives
\begin{equation}\label{eq:z2-m-upper}
0<m<2A,
\qquad
D<\frac{t(2A-m)}{2A}.
\end{equation}
Combining this with \eqref{eq:z2-D-lower} gives
\begin{equation}\label{eq:z2-x-lower}
x^2>\frac{m}{2A-m}.
\end{equation}
Consequently, every fully mixed anisotropic solution lies in the region
\begin{equation}\label{eq:z2-chamber}
2A<B+C<4A,
\qquad
X^2>\frac{B+C-2A}{4A-B-C}.
\end{equation}

\section{An exact square identity and anisotropy bound}\label{sec:z2-square}

Set
\begin{equation}\label{eq:z2-betavars}
\xi=x^2,
\qquad
\beta=b^2,
\qquad
\eta=e^2,
\qquad
\tau=t^2.
\end{equation}
Since $s=2A+m$,
\begin{equation}\label{eq:z2-4r1s1}
\begin{aligned}
4r_1&=(b-m)^2+(t+e)^2+(2A+m-b)^2\xi,\\
4s_1&=(b+m)^2+(t-e)^2+(2A+m+b)^2\xi.
\end{aligned}
\end{equation}
Equation \eqref{eq:z2-R-exact} is equivalent to
\begin{equation}\label{eq:z2-2xR}
2xR_0=\mathfrak N,
\qquad
\mathfrak N:=t[m+(2A+m)\xi]+be\,\xi.
\end{equation}
Define
\begin{equation}\label{eq:z2-LM}
\begin{aligned}
\mathfrak L&=4A^2\xi+4Am\xi-\beta\xi+\beta+m^2\xi+m^2-\tau,\\
\mathfrak M&=m^2(1+\xi)+4Am\xi-\beta\xi.
\end{aligned}
\end{equation}

\begin{lemma}[Polynomial square identity]\label{lem:z2-square}
The following identity holds identically in the variables $A,m,b,t,e,\xi$:
\begin{equation}\label{eq:z2-square-identity}
16\xi r_1s_1-4\mathfrak N^2
=\xi(\eta+\mathfrak L)^2-4(\tau+\beta\xi)\mathfrak M.
\end{equation}
\end{lemma}

\begin{proof}
Substitute \eqref{eq:z2-4r1s1} and the definition of $\mathfrak N$ into the left side.  The terms of odd degree in $e$ cancel.  Collecting the remaining terms gives
\[
16\xi r_1s_1-4\mathfrak N^2
=\xi\eta^2+2\xi\mathfrak L\eta+\xi\mathfrak L^2
-4(\tau+\beta\xi)\mathfrak M,
\]
which is \eqref{eq:z2-square-identity}.
\end{proof}

At a solution, $R_0^2=r_1s_1$ by definition and
$\mathfrak N=2xR_0$ by \eqref{eq:z2-2xR}.  Since $\xi=x^2$, the left-hand side of
\eqref{eq:z2-square-identity} vanishes.  Hence
\[
\xi(\eta+\mathfrak L)^2
=4(\tau+\beta\xi)\mathfrak M.
\]
Because $\xi>0$ and $\tau+\beta\xi>0$, one has $\mathfrak M\ge0$, and therefore
\begin{equation}\label{eq:z2-b-bound-raw}
b^2\le4Am+m^2+\frac{m^2}{x^2}.
\end{equation}
Using \eqref{eq:z2-x-lower},
\[
\frac1{x^2}<\frac{2A-m}{m},
\]
and thus
\begin{equation}\label{eq:z2-b-bound}
b^2<6Am<4A(A+m).
\end{equation}
This strict bound will make the final quadratic form positive definite.

\section{A diagonal Einstein difference with fixed sign}\label{sec:z2-difference}

Among the six diagonal Lagrange-multiplier equations in \eqref{eq:z2-diag-einstein}, the $B$- and $C$-equations require
\[
BS_B=CS_C.
\]
We will show that this equality is impossible on the fully mixed anisotropic BCHL branch by proving that, after multiplication by $B-C$, the difference $BS_B-CS_C$ has a definite negative sign.

Set
\begin{equation}\label{eq:z2-UBUC}
U_B=A-B(1+x^2),
\qquad
U_C=A-C(1+x^2),
\end{equation}
and
\begin{equation}\label{eq:z2-kell}
k=\frac{R_0}{r_1}=\sqrt{\frac{s_1}{r_1}}>0,
\qquad
\ell_0=\frac{V_0}{r_2}=\sqrt{\frac{s_2}{r_2}}>0.
\end{equation}
The paired equations and the sign conventions imply
\begin{equation}\label{eq:z2-amplitudes}
Y^2=k\rho,
\qquad
Z^2=\frac{\rho}{k},
\qquad
W^2=\ell_0\omega,
\qquad
\Th^2=\frac{\omega}{\ell_0}.
\end{equation}
A direct differentiation of \eqref{eq:z2-S} gives
\begin{align}
BS_B-CS_C
={}&AB-AC-(B^2-C^2)(1+X^2)\nonumber\\
&+B[A-B(1+X^2)](W^2+Z^2)\nonumber\\
&-C[A-C(1+X^2)](Y^2+\Th^2)\nonumber\\
&+X(BE-CF)YZ+X(CE-BF)W\Th.\label{eq:z2-raw-difference}
\end{align}
Substituting \eqref{eq:z2-sign-variables} and \eqref{eq:z2-amplitudes} yields
\begin{equation}\label{eq:z2-difference-decomp}
BS_B-CS_C
=b[A-s(1+x^2)]+\rho T_++\omega T_-,
\end{equation}
where
\begin{equation}\label{eq:z2-Tpm}
\begin{aligned}
T_+&=\frac{BU_B}{k}-CU_Ck+x(BE-CF),\\
T_-&=BU_B\ell_0-\frac{CU_C}{\ell_0}+x(BF-CE).
\end{aligned}
\end{equation}
Define
\begin{equation}\label{eq:z2-Hpm}
\begin{aligned}
H_-&=A(1+x^2)(m^2-b^2)-4A^3x^2-At^2-(A+m)e^2-bet,\\
H_+&=A(1+x^2)(m^2-b^2)-4A^3x^2-At^2-(A+m)e^2+bet.
\end{aligned}
\end{equation}

\begin{lemma}[Factorization of the difference terms]\label{lem:z2-Tfactor}
Using the relation \eqref{eq:z2-R-exact},
\begin{equation}\label{eq:z2-Tfactor}
4R_0T_+=bH_-,
\qquad
4V_0T_-=bH_+.
\end{equation}
\end{lemma}

\begin{proof}
Since $R_0/k=r_1$ and $R_0k=s_1$,
\[
R_0T_+=r_1BU_B-s_1CU_C+xR_0(BE-CF).
\]
Write
\[
B=\frac{2A+m+b}{2},\quad
C=\frac{2A+m-b}{2},\quad
E=\frac{t+e}{2},\quad
F=\frac{t-e}{2},
\]
and substitute \eqref{eq:z2-4r1s1} and \eqref{eq:z2-2xR}; collecting terms gives $4R_0T_+=bH_-$.  The calculation for $T_-$ is identical, with the sign of the term linear in $e$ reversed.
\end{proof}

\begin{lemma}\label{lem:z2-Hnegative}
On the fully mixed anisotropic BCHL branch,
\[
H_-<0,
\qquad
H_+<0.
\]
\end{lemma}

\begin{proof}
From \eqref{eq:z2-x-lower},
\begin{align*}
(1+x^2)m^2-4A^2x^2
&=m^2-x^2(4A^2-m^2)\\
&<m^2-\frac{m}{2A-m}(2A-m)(2A+m)\\
&=-2Am<0.
\end{align*}
Hence
\begin{equation}\label{eq:z2-first-negative}
A(1+x^2)(m^2-b^2)-4A^3x^2<0.
\end{equation}
For either sign,
\begin{equation}\label{eq:z2-quadratic-positive}
At^2+(A+m)e^2\pm bet
=A\left(t\pm\frac{be}{2A}\right)^2
+\left(A+m-\frac{b^2}{4A}\right)e^2.
\end{equation}
By \eqref{eq:z2-b-bound}, the second coefficient is strictly positive, and $e\neq0$.  Thus \eqref{eq:z2-quadratic-positive} is strictly positive.  Combining this with \eqref{eq:z2-first-negative} proves $H_\pm<0$.
\end{proof}

Multiplying \eqref{eq:z2-difference-decomp} by $b$ and using Lemma~\ref{lem:z2-Tfactor},
\begin{equation}\label{eq:z2-final-negative}
\begin{aligned}
b(BS_B-CS_C)
={}&b^2[A-s(1+x^2)]\\
&+\rho\frac{b^2H_-}{4R_0}
+\omega\frac{b^2H_+}{4V_0}.
\end{aligned}
\end{equation}
Every term on the right is strictly negative: $b^2>0$, $s>2A$, $\rho,\omega,R_0,V_0>0$, and $H_\pm<0$.  Therefore
\begin{equation}\label{eq:z2-strict-contradiction}
b(BS_B-CS_C)<0,
\end{equation}
contradicting $BS_B=CS_C$.  We have proved:

\begin{proposition}[No fully mixed anisotropic branch]\label{prop:z2-no-genuine}
Every trace $-2$, $\mathbb Z_2$-invariant Einstein solution satisfies
\begin{equation}\label{eq:z2-degenerate}
(YZ)(W\Th)(B-C)(E-F)=0.
\end{equation}
\end{proposition}

\section{Vanishing-parameter cases and the larger symmetry group}\label{sec:z2-upgrade}

By Proposition~\ref{prop:z2-no-genuine}, every solution has
$YZ=0$, $W\Th=0$, $B=C$, or $E=F$.
We call these the \emph{degenerate parameter cases}; the metric itself
remains positive definite.

\begin{proof}[Completion of the proof of Theorem~\ref{thm:z2upgrade}]
It remains to prove that every solution satisfying \eqref{eq:z2-degenerate} has at least $\mathbb Z_2^2$ inner symmetry.  After the initial inner change of basis, the assumed involution is $(R_1,R_1)$, where $R_1=\diag(1,-1,-1)$.  Every additional basis rotation used below lies in the oriented $2$--$3$ plane and hence commutes with this involution.

Write
\[
E_\perp=\binom{E_2}{E_3},
\qquad
F_\perp=\binom{F_2}{F_3},
\qquad
X_\perp=\binom{X_2}{X_3},
\qquad
Y_\perp=\binom{Y_2}{Y_3}.
\]
Set
\[
D_E=\begin{pmatrix}b_0&0\\0&c_0\end{pmatrix},
\qquad
D_F=\begin{pmatrix}e_0&0\\0&f_0\end{pmatrix}.
\]
Then the $2$--$3$ part of \eqref{eq:z2-normal} is
\[
X_\perp=D_EE_\perp,
\qquad
Y_\perp=NE_\perp+D_FF_\perp,
\]
where
\begin{equation}\label{eq:z2-N}
N=
\begin{pmatrix}y_0&w_0\\ \gamma_0&z_0\end{pmatrix}
=
\begin{pmatrix}
\sqrt{AC}\,Y&\sqrt{AB}\,W\\
\sqrt{AC}\,\Th&\sqrt{AB}\,Z
\end{pmatrix}.
\end{equation}
The larger $\mathbb Z_2^2$ symmetry becomes transparent when the $2\times2$ mixing block $N$ is diagonal.  We may act on the oriented bases $E_\perp,F_\perp$ by matrices $Q_E,Q_F\in SO(2)$; these are induced by inner automorphisms.  Independently, we may apply matrices $P_X,P_Y\in O(2)$ to the orthonormal frame pairs $X_\perp,Y_\perp$; this changes only the chosen orthonormal frame, not the metric.  Under
\[
E_\perp'=Q_EE_\perp,\quad F_\perp'=Q_FF_\perp,
\quad X_\perp'=P_XX_\perp,\quad Y_\perp'=P_YY_\perp,
\]
the three coefficient blocks become
\[
P_XD_EQ_E^T,\qquad P_YNQ_E^T,\qquad P_YD_FQ_F^T.
\]
We now choose these four matrices explicitly in each case.

If $YZ=0$, then \eqref{eq:z2-pair-zero} gives $Y=Z=0$, so
\[
N=\begin{pmatrix}0&w_0\\ \gamma_0&0\end{pmatrix}.
\]
Let
\[
J_2=\begin{pmatrix}0&-1\\1&0\end{pmatrix}\in SO(2).
\]
Replace the $E$-basis by $E_\perp'=J_2E_\perp$ and, simultaneously, the orthonormal frame pair by $X_\perp'=J_2X_\perp$.  Then
\[
X_\perp'=(J_2D_EJ_2^T)E_\perp',
\qquad
Y_\perp=(NJ_2^T)E_\perp'+D_FF_\perp.
\]
Here $J_2D_EJ_2^T=\diag(c_0,b_0)$ and $NJ_2^T=\diag(-w_0,\gamma_0)$, so both blocks are diagonal.  If $W\Th=0$, then \eqref{eq:z2-pair-zero} gives $W=\Th=0$, and $N$ is already diagonal.

Assume now $YZ\neq0$ and $W\Th\neq0$.  Then \eqref{eq:z2-degenerate} gives either $B=C$ or $E=F$.

If $B=C$, then $u+v=0$ and $r_1=r_2$.  The paired equations give
\[
r_1(Y\Th+WZ)=(r_1Y)\Th+(r_2W)Z=uZ\Th+v\Th Z=0,
\]
so
\begin{equation}\label{eq:z2-row-orthogonal}
Y\Th+WZ=0.
\end{equation}
Because $B=C$, the two rows of $N$ are orthogonal.  Writing $N_{1*}$ and $N_{2*}$ for the first and second rows,
\[
\ip{N_{1*}}{N_{2*}}=AB(Y\Th+WZ)=0.
\]
Choose $Q_2\in SO(2)$ such that $NQ_2^T$ is diagonal.  Since $B=C$ is equivalent to $D_E=b_0I_2$, replace $E_\perp$ by $E_\perp'=Q_2E_\perp$ and the orthonormal frame pair by $X_\perp'=Q_2X_\perp$.  Then
\[
X_\perp'=b_0E_\perp',
\qquad
Y_\perp=(NQ_2^T)E_\perp'+D_FF_\perp,
\]
so the required diagonal form is preserved in the first block and achieved in the mixing block.

If $E=F$, then $u+v=0$ and $r_1=s_2$.  The paired equations similarly give
\begin{equation}\label{eq:z2-column-orthogonal}
YW+\Th Z=0.
\end{equation}
Hence the columns of $N$ are orthogonal.  Choose $P_2\in SO(2)$ such that $P_2N$ is diagonal.  Since $E=F$ is equivalent to $D_F=e_0I_2$, replace the orthonormal frame pair by $Y_\perp'=P_2Y_\perp$ and the $F$-basis by $F_\perp'=P_2F_\perp$.  Then
\[
Y_\perp'=(P_2N)E_\perp+e_0F_\perp',
\]
again giving diagonal blocks.

Thus every degenerate branch is equivalent, by inner automorphisms and orthogonal changes of orthonormal frame, to
\begin{equation}\label{eq:z2x2-normal}
(X_1,X_2,X_3,Y_1,Y_2,Y_3)
=(a_0E_1,b_0E_2,c_0E_3,
 x_0E_1+d_0F_1,
 y_0E_2+e_0F_2,
 z_0E_3+f_0F_3),
\end{equation}
with possibly relabelled positive diagonal coefficients.  Let
\[
R_1=\diag(1,-1,-1),\quad
R_2=\diag(-1,1,-1),\quad
R_3=\diag(-1,-1,1).
\]
Then
\[
\Gamma_0=\{(I,I),(R_1,R_1),(R_2,R_2),(R_3,R_3)\}
\subset\Ad(G)
\]
preserves \eqref{eq:z2x2-normal} and is isomorphic to $\mathbb Z_2^2$: each of its elements sends every vector in the displayed orthonormal frame to that same vector up to sign.  It contains the normalized form $(R_1,R_1)$ of the original involution.  Conjugating back by the inner automorphisms used in the reduction therefore gives a Klein four subgroup containing $\sigma$, as asserted.
\end{proof}

\section{Synthesis: proof of the classification theorem}\label{sec:synthesis}

\begin{proof}[Proof of Theorem~\ref{thm:classification}]
Let $g$ be any left-invariant Einstein metric on $G$.  Theorem~\ref{thm:auto} shows first that $K(g)$ contains a nonidentity element, so the case of trivial inner isotropy does not occur.

Because $K(g)$ is a compact Lie subgroup of $\Ad(G)$, it is either finite or positive-dimensional.  If it has positive dimension, its identity component contains a circle subgroup, and \cite[Theorem~1]{BCHL} gives, up to homothety and isometry, $g_{\rm can}$ or $g_{\rm NK}$.  It remains to consider finite $K(g)$.  If this finite group is not isomorphic to a single copy of $\mathbb Z_2$, \cite[Theorem~2]{BCHL}, applied with $\Gamma=K(g)$, again gives the same two metrics.  Therefore the only case not yet covered by those results is
\[
K(g)\cong\mathbb Z_2.
\]

Let $\sigma$ be the nonidentity element of this group.  As noted above, a nonidentity involution in $SO(3)\times SO(3)$ has trace either $2$ or $-2$.  The trace $2$ case is classified in \cite[Proposition~7]{BCHL}.  If $\tr\sigma=-2$, Theorem~\ref{thm:z2upgrade} shows that $g$ is in fact invariant under a Klein four subgroup.  The metric therefore falls under \cite[Theorem~2]{BCHL}, giving the same classification.

All possibilities for $K(g)$ have now been exhausted, which proves the theorem.
\end{proof}

\appendix
\section*{Acknowledgements}
The author Sixuan Gu is simultaneously supervised by Professor Jiu-Kang Yu of Hetao Institute of Mathematics and Interdisciplinary Sciences and Professor Caihua Luo of the Chinese University of Hong Kong (Shenzhen), and is funded by the latter during this research.  
\section*{Declaration of generative AI and AI-assisted technologies in the manuscript preparation process}

During the preparation of this work, the author used OpenAI's ChatGPT to explore proof strategies and assist in the drafting. The author reviewed and edited the resulting material and takes full responsibility for the content of the publication.

\section{Dictionary with the explicit BCHL scalar-curvature polynomial}\label{app:dictionary}

For the graph-coordinate formula \eqref{eq:scalar}, take
\[
P=\diag(A,B,C),
\qquad
Q=\diag(D,E,F),
\qquad
M=
\begin{pmatrix}
X&U&V\\
\mathcal A&Y&W\\
\mathcal B&\mathcal C&Z
\end{pmatrix}.
\]
Then
\[
(\cof P)^{1/2}=\diag(\sqrt{BC},\sqrt{AC},\sqrt{AB}),
\]
and the lower-left block in the explicit coordinates of \cite{BCHL} is
\[
W_{\rm BCHL}=M(\cof P)^{1/2}.
\]
The terms depending only on $P$ are
\[
AB+AC+BC-\frac12(A^2+B^2+C^2)=\Phi(P),
\]
and similarly for $Q$.  The coefficient of the squared norm of the $i$th column of $M$ is
\[
p_jp_k-\frac12(p_j^2+p_k^2)
=-\frac12(p_j-p_k)^2,
\]
so the column terms sum to
\[
-\frac12\tr(M^TM\Lop(P)).
\]
Finally, the remaining nine squares are 
\[
\bigl(p_j(\cof M)_{ij}-q_iM_{ij}\bigr)^2,
\]
the nine component squares of
\[
R=(\cof M)P-QM.
\]
Thus, after the displayed substitution, the coordinate-free formula \eqref{eq:scalar} is algebraically equivalent term by term to the explicit polynomial in \cite{BCHL}.

\section{Polynomial form of the trace \texorpdfstring{$-2$}{-2} system}\label{app:z2-system}

For comparison with \cite[Eq. (3.18)]{BCHL}, the normalized-volume constraint and diagonal equations are
\begin{align*}
0={}&ABCDEF-1,\\
0={}&-D+E+F+ABCEF\mu-DX^2+AXYZ-AWX\Th,\\
0={}&D-E+F+ABCDF\mu-EW^2-EY^2+BXYZ-CWX\Th,\\
0={}&A-B+C+ACDEF\mu+AW^2-BW^2-BX^2+CX^2-BW^2X^2\\
&\quad+EXYZ+AZ^2-BZ^2-BX^2Z^2-FWX\Th,\\
0={}&D+E-F+ABCDE\mu+CXYZ-FZ^2-BWX\Th-F\Th^2,\\
0={}&-A+B+C+BCDEF\mu-AW^2+BW^2-AY^2+CY^2+DXYZ\\
&\quad-AZ^2+BZ^2-AY^2Z^2-DWX\Th+2AWYZ\Th-A\Th^2+C\Th^2-AW^2\Th^2,\\
0={}&A+B-C+ABDEF\mu+BX^2-CX^2+AY^2-CY^2-CX^2Y^2+FXYZ\\
&\quad-EWX\Th+A\Th^2-C\Th^2-CX^2\Th^2.
\end{align*}
The five off-diagonal equations are
\begin{align*}
0={}&-ADWX-CEWX-BFWX+A^2WYZ-A^2\Th+2AC\Th-C^2\Th-F^2\Th\\
&\quad-A^2W^2\Th-C^2X^2\Th,\\
0={}&ADXY+BEXY+CFXY-A^2Z+2ABZ-B^2Z-F^2Z-B^2X^2Z\\
&\quad-A^2Y^2Z+A^2WY\Th,\\
0={}&-A^2Y+2ACY-C^2Y-E^2Y-C^2X^2Y+ADXZ+BEXZ+CFXZ\\
&\quad-A^2YZ^2+A^2WZ\Th,\\
0={}&-A^2W+2ABW-B^2W-E^2W-B^2WX^2-ADX\Th-CEX\Th-BFX\Th\\
&\quad+A^2YZ\Th-A^2W\Th^2,\\
0={}&-B^2X+2BCX-C^2X-D^2X-B^2W^2X-C^2XY^2+ADYZ+BEYZ+CFYZ\\
&\quad-B^2XZ^2-ADW\Th-CEW\Th-BFW\Th-C^2X\Th^2.
\end{align*}
Collecting the last five equations with respect to $Y,Z,W,\Th,X$ gives precisely \eqref{eq:z2-pair-eqs} and \eqref{eq:z2-X-eq}.

\end{document}